\documentclass[a4paper,12pt]{article}
\title{{\bf Curve counting invariants around the \\
conifold point}}
\date{}
\author{Yukinobu Toda}

\usepackage{makeidx}

\usepackage{latexsym}
\usepackage{amscd}
\usepackage{amsmath}
\usepackage{amssymb}
\usepackage{amsthm}
\usepackage{float}
\usepackage[dvips]{graphicx}

\usepackage[all,ps,dvips]{xy}

\usepackage{array}
\usepackage{amscd}
\usepackage[all]{xy}
\usepackage{makeidx}
\usepackage{latexsym}
\DeclareFontFamily{U}{rsfs}{%
\skewchar\font127}
\DeclareFontShape{U}{rsfs}{m}{n}{%
<-6>rsfs5<6-8.5>rsfs7<8.5->rsfs10}{}
\DeclareSymbolFont{rsfs}{U}{rsfs}{m}{n}
\DeclareSymbolFontAlphabet
{\mathrsfs}{rsfs}
\DeclareRobustCommand*\rsfs{%
\@fontswitch\relax\mathrsfs}
\theoremstyle{plain}
\newtheorem{thm}{Theorem}[section]
\newtheorem{prop}[thm]{Proposition}
\newtheorem{lem}[thm]{Lemma}

\newtheorem{defi}[thm]{Definition}
\newtheorem{rmk}[thm]{Remark}

\newtheorem{step}{Step}

\newtheorem{prop-defi}[thm]{Proposition-Definition}
\newtheorem{thm-defi}[thm]{Theorem-Definition}
\newtheorem{lem-defi}[thm]{Lemma-Definition}

\newtheorem{question}[thm]{Question}

\newtheorem{conj}[thm]{Conjecture}

\newdimen\argwidth
\def\db[#1\db]{
 \setbox0=\hbox{$#1$}\argwidth=\wd0
 \setbox0=\hbox{$\left[\box0\right]$}
  \advance\argwidth by -\wd0
 \left[\kern.3\argwidth\box0 \kern.3\argwidth\right]}

\newcommand{\aA}{\mathcal{A}}
\newcommand{\bB}{\mathcal{B}}
\newcommand{\cC}{\mathcal{C}}
\newcommand{\dD}{\mathcal{D}}
\newcommand{\eE}{\mathcal{E}}
\newcommand{\fF}{\mathcal{F}}

\newcommand{\hH}{\mathcal{H}}

\newcommand{\mM}{\mathcal{M}}

\newcommand{\oO}{\mathcal{O}}
\newcommand{\pP}{\mathcal{P}}
\newcommand{\qQ}{\mathcal{Q}}

\newcommand{\sS}{\mathcal{S}}
\newcommand{\tT}{\mathcal{T}}
\newcommand{\uU}{\mathcal{U}}

\newcommand{\xX}{\mathcal{X}}
\newcommand{\yY}{\mathcal{Y}}
\newcommand{\zZ}{\mathcal{Z}}

\newcommand{\Supp}{\mathop{\rm Supp}\nolimits}
\newcommand{\Hom}{\mathop{\rm Hom}\nolimits}

\newcommand{\dR}{\mathbf{R}}

\newcommand{\Pic}{\mathop{\rm Pic}\nolimits}

\newcommand{\id}{\textrm{id}}

\newcommand{\ch}{\mathop{\rm ch}\nolimits}

\newcommand{\Ext}{\mathop{\rm Ext}\nolimits}
\newcommand{\Spec}{\mathop{\rm Spec}\nolimits}

\newcommand{\Coh}{\mathop{\rm Coh}\nolimits}

\newcommand{\cneq}{\mathrel{\raise.095ex\hbox{:}\mkern-4.2mu=}}
\newcommand{\eqcn}{\mathrel{=\mkern-4.5mu\raise.095ex\hbox{:}}}

\newcommand{\gr}{\mathop{\rm gr}\nolimits}

\newcommand{\Aut}{\mathop{\rm Aut}\nolimits}

\newcommand{\Stab}{\mathop{\rm Stab}\nolimits}

\newcommand{\DT}{\mathop{\rm DT}\nolimits}

\newcommand{\Eu}{\mathop{\rm Eu}\nolimits}

\newcommand{\Imm}{\mathop{\rm Im}\nolimits}

\newcommand{\Ree}{\mathop{\rm Re}\nolimits}

\newcommand{\GL}{\mathop{\rm GL}\nolimits}

\newcommand{\tr}{\mathop{\rm tr}\nolimits}
\newcommand{\ex}{\mathop{\rm ex}\nolimits}

\newcommand{\cl}{\mathop{\rm cl}\nolimits}
\begin{document}
\maketitle
\begin{abstract}
In this paper, we investigate the space of
certain 
weak stability 
conditions on the triangulated
category of D0-D2-D6 bound states
on a smooth projective Calabi-Yau 3-fold. 
In the case of a quintic 3-fold, 
the resulting space is interpreted as 
a universal covering space of 
an infinitesimal neighborhood
of the conifold point in the stringy
K$\ddot{\textrm{a}}$hler moduli space. 
We then construct the DT type invariants 
counting semistable objects in 
our triangulated category, which are 
new curve counting invariants 
on a Calabi-Yau 3-fold. 
We also 
investigate 
the wall-crossing formula
of our invariants 
 and their
interplay with the Seidel-Thomas twist. 
\end{abstract}
\section{Introduction}
\subsection{Motivation}
Let $X$ be a smooth projective Calabi-Yau 
3-fold over $\mathbb{C}$, i.e. 
\begin{align*}
\bigwedge^{3}T_X^{\vee} \cong \oO_X, \quad
H^1(X, \oO_X)=0. 
\end{align*}
So far several 
curve counting theories on $X$
have been introduced and studied:
\begin{itemize}
\item {\bf Gromov-Witten (GW) theory~\cite{KonMa}:} counting 
stable maps $f\colon C\to X$
from projective nodal curves $C$. 
\item {\bf Donaldson-Thomas (DT) theory~\cite{Thom}:}
counting 1-dimensional 
subschemes $C\subset X$. 
\item {\bf Pandharipande-Thomas (PT) theory~\cite{PT}:} 
counting stable pairs $(F, s)$. 
Here $F$ is a 1-dimensional pure
sheaf and $s$ is a morphism 
$s\colon \oO_X \to F$ with 0-dimensional 
cokernel.  
\end{itemize}
The above theories are 
conjecturally equivalent
in terms of generating functions. 
The GW/DT correspondence is 
proved for local toric Calabi-Yau 3-folds~\cite{MNOP},
local curves~\cite{BrPa}, \cite{OkPa}, and 
the DT/PT correspondence
(including the Euler characteristic version) is 
available in~\cite{StTh}, \cite{BrH}, \cite{Tcurve1}. 

The idea of DT/PT correspondence discussed
by Pandharipande-Thomas~\cite{PT} is to use the wall-crossing formula 
of DT type invariants 
in the space of Bridgeland's stability conditions~\cite{Brs1}
on the category $D^b(\Coh(X))$, the
bounded derived category of coherent sheaves on $X$.
Namely it is expected that there are 
two stability conditions $\sigma$, $\tau$
on $D^b(\Coh(X))$ such that 
the moduli space of $\sigma$-stable 
objects and that of $\tau$-stable 
objects 
with a certain numerical condition
coincide with the moduli spaces
which define DT and PT theories respectively. 
Then DT/PT correspondence should follow 
by investigating the behavior of 
the invariants under the change of 
stability conditions. 
A general framework
for such a study, known as a 
\textit{wall-crossing formula of DT type invariants}, 
is now established by the work of Joyce-Song~\cite{JS}
and Kontsevich-Soibelman~\cite{K-S}. 

However there have been difficulties in constructing 
stability conditions on $D^b(\Coh(X))$, 
and even a single example is not available yet. 
Instead of working with Bridgeland's stability conditions,   
Bayer~\cite{Bay} and the 
author~\cite{Tolim}, \cite{Tcurve1} independently 
introduce \textit{polynomial stability}, \textit{limit stability}
and \textit{weak stability} respectively. These notions are 
interpreted as `limiting degenerations' of Bridgeland's
stability conditions. 
By using the above degenerated stability conditions, 
it is turned out in~\cite{Bay}, \cite{Tcurve1}
 that DT/PT correspondence 
is realized near a particular point, 
called the \textit{large volume limit}. 
By analyzing 
weak stability conditions
near the large volume limit 
and the relevant wall-crossing formula, 
the author
proves the
Euler characteristic 
version of
DT/PT correspondence~\cite{Tcurve1}
and the rationality conjecture of the
 generating series of DT and PT invariants~\cite{Tolim2}. 

The space of stability conditions
on $D^b(\Coh(X))$ is
expected to be related to the stringy 
 K$\ddot{\textrm{a}}$hler moduli space of $X$, 
which is the moduli space of complex structures 
of a mirror manifold. 
For instance if $X$ is a quintic Calabi-Yau 
3-fold in $\mathbb{P}^4$, 
the mirror family is a simultaneous 
crepant resolution $\widehat{Y}_{\psi}$
of the following 
one parameter family, 
\begin{align*}
Y_{\psi}=\left\{\sum_{i=0}^{4}y_i^5 -5\psi \prod_{i=0}^4 y_i =0
\right\}/G
\subset \mathbb{P}^4/G, 
\end{align*}
where $G=(\mathbb{Z}/5\mathbb{Z})^{3}$. 
The stringy K$\ddot{\textrm{a}}$hler moduli 
space is a parameter space of $\psi^5$, and 
the large volume limit corresponds to $\psi^5=\infty$. 
(See Figure~\ref{fig:one}.)

So far the above studies on curve counting 
theories appear near the large volume limit. Now it is natural 
to address the following question. 
\begin{question}\label{Quest}
{\bf What kinds of curve counting invariants
(or DT type invariants) appear at 
other limiting points?}
\end{question}
In this paper, we study the above
question for another limiting point, 
called \textit{conifold point}. 
In the case of a quintic 3-fold, 
this point corresponds to $\psi^5=1$.
There is a  Lagrangian sphere 
in a mirror manifold 
$\widehat{Y}_{\psi}$
which vanishes at the conifold point,  
and it corresponds to the object $\oO_X$ 
under the mirror symmetry.
In physic terminology,
the \textit{mass}
of the object $\oO_X$, denoted by 
$m(\oO_X)$, behaves as 
\begin{align}\label{mO}
m(\oO_X) \to \left\{ \begin{array}{ll}
\infty, & \mbox{ at large volume limit, }\\
0, & \mbox{ at conifold point. }
\end{array} \right. 
\end{align} 
Namely the conifold point 
is a point where the object $\oO_X$
becomes massless, and its effect
to the stability should be 
infinitesimally small. 
On the other hand, there is an 
autoequivalence associated with $\oO_X$, 
called \textit{Seidel-Thomas twist}~\cite{ST}, 
\begin{align*}
\Phi_{\oO_X}\colon D^b(\Coh(X)) \stackrel{\sim}{\to}
D^b(\Coh(X)).
\end{align*}
The above equivalence should
correspond to the Dehn twist of $\widehat{Y}_{\psi}$
along the Lagrangian vanishing
cycle under the mirror symmetry, 
and should be a monodromy on 
$D^b(\Coh(X))$ 
around the conifold point.
Therefore the Seidel-Thomas
twist must be relevant in
studying Question~\ref{Quest}
at the conifold point, 
and it seems interesting to see
how the twist functor is related to the wall-crossing formula.  
\begin{figure}[htbp]
 \begin{center}
  \includegraphics[width=75mm]{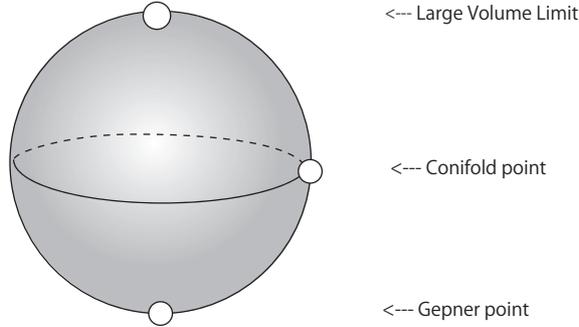}
 \end{center}
 \caption{Stringy K$\ddot{\textrm{a}}$hler moduli space
 of a quintic 3-fold}
 \label{fig:one}
\end{figure}

\subsection{Weak stability conditions on 
D0-D2-D6 bound states}
In this paper, we focus on 
 the triangulated
category
called \textit{D0-D2-D6 bound states}
\begin{align}\label{026}
\dD_X =\langle \oO_X, \Coh_{\le 1}(X) \rangle_{\tr}
\subset D^b(\Coh(X)). 
\end{align}
This is the smallest triangulated subcategory 
of $D^b(\Coh(X))$ which contains
$\oO_X$ and the objects in $\Coh_{\le 1}(X)$, 
\begin{align*}
\Coh_{\le 1}(X)=\{E\in \Coh(X):
\dim \Supp(E) \le 1\}. 
\end{align*} 
The category $\dD_X$ is especially important in studying 
curve counting invariants on $X$. For instance it contains 
ideal sheaves of curves,
 two term complexes associated with stable pairs, 
and DT/PT correspondence is realized there~\cite{Tcurve1}.

We use the following finitely 
generated abelian group $\Gamma$, 
 \begin{align*}
\Gamma=H^0(X, \mathbb{Z}) \oplus H_2(X, \mathbb{Z}) \oplus 
H_0(X, \mathbb{Z}), 
\end{align*}
which is the image of the Chern 
character map from the category $\dD_X$. 
Roughly speaking, a Bridgeland's stability 
condition on $\dD_X$ consists of 
data $(Z, \pP)$, 
\begin{align*}
Z \colon \Gamma \to \mathbb{C}, \quad 
\pP(\phi) \subset \dD_X, 
\end{align*}
where $Z$ is a group homomorphism and $\pP(\phi)$
is a full subcategory for $\phi\in \mathbb{R}$, 
which satisfy some axiom. 

The notion of weak stability conditions on $\dD_X$
is determined after we specify a filtration
$\Gamma_{\bullet}$ of $\Gamma$, 
(cf.~Definition~\ref{defi:weak},)
which is to do with the limiting direction of
Bridgeland stability. 
The set of weak stability conditions is denoted by,
\begin{align}\label{setwe}
\Stab_{\Gamma_{\bullet}}(\dD_X),
\end{align}
and it has a natural structure of a complex manifold. 
(cf.~Theorem~\ref{thm:top}.)

In this paper, we are interested in the following filtration, 
\begin{align}\label{filt1}
\Gamma_{0} =H^0(X, \mathbb{Z}) \subset \Gamma_1=\Gamma.
\end{align}
A point in the space (\ref{setwe})
w.r.t. the filtration (\ref{filt1}) 
corresponds to  
 a limiting degeneration of Bridgeland stability $(Z, \pP)$
on $\dD_X$, 
whose limiting direction is given by the constraint, 
\begin{align*}
\lvert Z(\ch(\oO_X)) \rvert  \ll \lvert Z(\ch(F)) \rvert,
\end{align*} 
for any
non-zero object $F\in \Coh_{\le 1}(X)$. 
This means that the effect of $\oO_X$
is infinitesimally small
relative to the objects in $\Coh_{\le 1}(X)$, hence 
the space (\ref{setwe})
seems to be related to 
an infinitesimal neighborhood 
of the conifold point
in Figure~\ref{fig:one}. 
In fact we have the following result. 
\begin{thm}\label{main:stab}\emph{\bf{[Theorem~\ref{obtain}]}}
Suppose that $H^2(X, \mathbb{Z}) \cong \mathbb{Z}$. 
(e.g. quintic 3-fold.)
Then there is a connected component 
\begin{align*}
\Stab_{\Gamma_{\bullet}}^{\circ}(\dD_X)\subset 
\Stab_{\Gamma_{\bullet}}(\dD_X)
\end{align*}
such that there is an isomorphism, 
\begin{align}\label{isom0}
\Stab_{\Gamma_{\bullet}}^{\circ}(\dD_X)
\cong \mathbb{C}\times \widetilde{\GL}_{+}(2, \mathbb{R}),
\end{align}
where $\widetilde{\GL}_{+}(2, \mathbb{R})$
is the universal cover of $\GL_{+}(2, \mathbb{R})$. 
The Seidel-Thomas twist $\Phi_{\oO_X}$ acts on 
the space (\ref{isom0}), and we have the isomorphism, 
\begin{align*}
\langle \Phi_{\oO_X} \rangle \backslash 
\Stab_{\Gamma_{\bullet}}^{\circ}(\dD_X) /
\mathbb{C} \cong  \mathbb{C}^{\ast} \times \mathbb{H}^{\circ}.
\end{align*}
Here $\langle \Phi_{\oO_X} \rangle$ is the subgroup of the 
group of autoequivalences of $\dD_X$ generated by $\Phi_{\oO_X}$, 
and $\mathbb{H}^{\circ}=\{z\in \mathbb{C} : \Imm z>0 \}$.  
\end{thm} 
Applying Theorem~\ref{main:stab}, we 
will construct a commutative diagram, 
(cf.~Subsection~\ref{sub:loop},)
\begin{align}\label{com:dia}
\xymatrix{
 \mathbb{R} \ar[r] \ar[d]_{\exp(\pi i \ast)} & 
\Stab_{\Gamma_{\bullet}}^{\circ}(\dD_X) \ar[d], \\
S^1 \ar[r]^{\iota} &  \mathbb{C}^{\ast} \times \mathbb{H}^{\circ},
}
\end{align}
where 
$\iota$ is an embedding
of $S^1$ to $(\mbox{unit circle}) \times 
\{\sqrt{-1} \}$. 
When $X$ is a quintic 3-fold, the image of $\iota$
may be interpreted as a loop around the conifold point
in Figure~\ref{fig:one}, 
since the monodromy around it is given by 
$\Phi_{\oO_X}$.

\subsection{DT theory around the conifold point}
Similarly to~\cite{Tcurve1}, \cite{Tcurve2}, 
we construct DT type invariants
counting semistable objects in $\dD_X$, 
 and investigate their wall-crossing phenomena. 
 In order to formulate the result, 
we denote the top arrow in the diagram (\ref{com:dia})
by $\gamma$, 
\begin{align*}
\gamma \colon \mathbb{R} \ni t \mapsto (Z_t, \pP_t) \in 
\Stab_{\Gamma_{\bullet}}^{\circ}(\dD_X). 
\end{align*}
For a data, 
\begin{align*}
(r, \beta, n)\in H^{0} \oplus H_2 \oplus H_0, \quad 
t \in \mathbb{R}, \quad \phi \in \mathbb{R}, 
\end{align*}
we will construct the generalized 
DT invariant, (cf.~Definition~\ref{def:around},)
\begin{align}\label{DTin}
\DT_{t}(r, \beta, n, \phi) \in \mathbb{Q},
\end{align}
following the construction by Joyce-Song~\cite{JS}. 
The invariant (\ref{DTin})
 counts objects $E\in \pP_t(\phi)$
satisfying the numerical condition, 
\begin{align*}
(\ch_0(E), \ch_2(E), \ch_3(E))
=(r, -\beta, -n). 
\end{align*}
The generating series $\DT_{t}(\phi)$ is defined by 
\begin{align}\label{DTt}
\DT_{t}(\phi)=
\sum_{(r, n, \beta)\in \Gamma}
\DT_{t}(r, \beta, n, \phi)x^r y^{\beta}z^n. 
\end{align}
The wall-crossing formula by Joyce-Song~\cite{JS}
and Kontsevich-Soibelman~\cite{K-S}
enables us to see how $\DT_t(\phi)$
changes under the change of $t$. 
Assuming a technical result 
announced by Behrend-Getzler~\cite{BG}, 
(cf.~Conjecture~\ref{conjBG},)
we will show the following result. 
\begin{thm}\emph{\bf{[Lemma~\ref{sense}, Theorem~\ref{thm:wall}]}}
\label{intro:thm}

(i) For a given $k\in \mathbb{Z}$, the series
$\DT_{t}(\phi)$ does not depend on a choice 
of $t\in (\phi+k, \phi+k+1)$. 
In particular, we may write it as $\DT^{k}(\phi)$. 

(ii) The series
$\DT^{k}(\phi)$ is obtained from $\DT^{k-1}(\phi)$
by the following transformation, 
\emph{\begin{align*}
z^n \mapsto \left\{ \begin{array}{cc}
(1-(-1)^{n}x)^{n}z^n, & \mbox{ if }k \mbox{ is even.} \\
x^n z^n/(1-(-1)^n x)^{n}, & \mbox{ if }k\mbox{ is odd.}
\end{array} \right. 
\end{align*}}
\end{thm}
The above theorem implies that 
the series $\DT^{k}(\phi)$ is obtained from 
$\DT^{k-2}(\phi)$
by the variable change 
$z\mapsto xz$, which coincides with 
the variable change by the Seidel-Thomas
twist $\Phi_{\oO_X}$. This means that, 
unfortunately, 
the wall-crossing formula 
does not provide any information on the invariant (\ref{DTin}), 
e.g. modularity. (cf.~Remark~\ref{modul}.)

On the other hand, the above theorem 
can be use to 
compute the series (\ref{DTt})
for a general $t$
if we know it for one point $t\in \mathbb{R}$
with $t\notin \mathbb{Z}+\phi$. 
For instance we will see that 
\begin{align*}
\DT_{t}(1)=-\chi(X)\sum_{\begin{subarray}{c}
n\ge 1, m\ge 1, \\
m|n. 
\end{subarray}}
\frac{1}{m^2}z^n, 
\end{align*}
when $0<t<1$
in Subsection~\ref{subsec:D0D6}. 
Applying Theorem~\ref{intro:thm}, 
we can write down the series (\ref{DTt})
for $\phi=1$ and a general $t$.
(cf.~Theorem~\ref{thm:DT}.) 
\subsection{Invariants on a local $(-1, -1)$-curve}
In Subsection~\ref{-1-1}, we 
focus on the invariants on a 
local $(-1, -1)$ curve, 
and especially investigate what kinds of objects
the invariants (\ref{DTin}) count.  
Let $C\subset X$ be an exceptional 
locus of a crepant small 
resolution of an ordinary double point, 
$f\colon X \to Y$. 
It satisfies that  
\begin{align*}
\mathbb{P}^1 \cong 
C\subset X, \quad N_{C/X} \cong \oO_C(-1) \oplus \oO_C(-1). 
\end{align*}
For instance we will see that   
the invariant 
\begin{align}\label{DTin2}
\DT_{t=1}(r, m[C], n, \phi) \in \mathbb{Q}, \quad
1/2<\phi<1,
\end{align} 
is non-zero only if $n=ma$ for some $a\in \mathbb{Z}_{\ge 1}$. 
In this case, the invariant (\ref{DTin2}) counts two term complexes, 
\begin{align*}
\oO_X^{\oplus r} \stackrel{s}{\to} \oO_C(a-1)^{\oplus m}, 
\end{align*}
such that the induced morphism,
\begin{align*}
H^0(s) \colon \mathbb{C}^{r} \to H^0(C, \oO_C(a-1))^{\oplus m},
\end{align*}
is injective. 
Applying Theorem~\ref{intro:thm}, we will 
compute the generating series of our invariants
in this situation. 
(cf.~Theorem~\ref{thm:ODP}.)
It is turned out that
there is a curious phenomena for the rank 
one generating series:
it coincides with the logarithm of  
the generating series of stable 
pairs on a local $(-1, -1)$-curve.
(cf.~Equation~(\ref{curious}).) 
Under GW/DT/PT correspondence, this 
implies that our invariants relate 
to \textit{connected} GW theory, while 
stable pair theory is related to 
non-connected GW theory. 
It seems interesting to give a geometric 
understanding of this phenomena.

\subsection{Relation to existing works}
Several examples of stability conditions 
have been studied in the 
literature, for instance~\cite{Brs2}, \cite{Brs3}, 
\cite{Brs4}, \cite{Tho}, \cite{Mac}, \cite{HMS}, \cite{BaMa}, 
\cite{IUU}, \cite{Tst}, \cite{Tst2}. 
However global descriptions of the spaces of degenerated
stability conditions introduced in~\cite{Bay}, \cite{Tolim}, 
\cite{Tcurve1} have not been 
studied so far. 
The result of Theorem~\ref{main:stab} is a first
example of such a study. 

The weak stability conditions on $\dD_X$ are also 
studied in~\cite{Tcurve1}, \cite{Tcurve2}
without giving any global descriptions of the 
spaces of weak stability conditions. 
We note that 
the filtrations taken in these works 
are different from (\ref{filt1}), 
and the
the resulting spaces should 
correspond to infinitesimal neighborhoods
of the large volume limit. 
It is worth mentioning that the 
equivalence $\Phi_{\oO_X}$
does \textit{not} act on 
the spaces discussed in~\cite{Tcurve1}, \cite{Tcurve2}. 

In~\cite{Tcurve1}, \cite{Tcurve2}, we also 
investigate the wall-crossing formula of 
DT type invariants with respect to 
certain weak stability conditions on $\dD_X$.
 However we focus
only on the rank one case in these works. 
Since the Seidel-Thomas twist is relevant 
in the study of our invariants, and 
the twist functor changes the rank, it is 
natural to consider higher rank invariants
in our situation. In fact we 
observe in Theorem~\ref{intro:thm} that the 
wall-crossing formula can be described 
only when we consider the generating 
series of all rank. 

Recently there have been studies on 
higher rank DT type 
invariants~\cite{Stop}, 
\cite{Trk2}, \cite{CDP}, \cite{Nhig}, \cite{WeiQin}. 
Since our study naturally involves higher 
rank invariants, it seems interesting to 
see a relationship to these works.

\subsection{Notation and convention}
For a variety $X$, the category of 
coherent sheaves on $X$ is denoted by $\Coh(X)$, 
and its derived category is denoted
by $D^b(\Coh(X))$. 
We use the following abelian subcategories of 
$\Coh(X)$, 
\begin{align*}
\Coh_{\le 1}(X) &=\{ E\in \Coh(X) : \dim \Supp(E) \le 1\}, \\
\Coh_{0}(X) &=\{ E\in \Coh(X) : \dim \Supp(E)=0\}. 
\end{align*}
For a triangulated category $\dD$ and a 
set of objects $S\subset \dD$, the 
subcategory $\langle S \rangle_{\tr} \subset \dD$
is the smallest triangulated subcategory 
of $\dD$ which contains objects in $S\cup \{0\}$. 
Also the category $\langle S \rangle_{\ex} \subset \dD$
is the smallest extension-closed subcategory 
which contains objects in $S\cup \{0\}$. 
If $S$ is a set of objects in an abelian category $\aA$, 
the subcategory $\langle S \rangle_{\ex} \subset \aA$
is also defined in a similar way. 
The varieties in this paper are defined 
over $\mathbb{C}$. 
For a variety $X$, we occasionally write 
$H^i(X, \mathbb{Z})$, $H_i(X, \mathbb{Z})$
as $H^i$, $H_i$ for simplicity.

\subsection{Acknowledgement}
The author thanks Kentaro Hori and 
Tom Bridgeland for 
valuable discussions. 
He also thanks Hokuto Uehara for 
the comments on the manuscript. 
This work is supported by World Premier 
International Research Center Initiative
(WPI initiative), MEXT, Japan. This work is also supported by Grant-in Aid
for Scientific Research grant (22684002), 
and partly (S-19104002),
from the Ministry of Education, Culture,
Sports, Science and Technology, Japan.

\section{The space of weak stability conditions}
The notion of weak stability conditions 
on triangulated categories is introduced 
in~\cite{Tcurve1} to give limiting degenerations 
of Bridgeland's stability conditions~\cite{Brs1}. 
In this section, we investigate 
the space of weak stability conditions on 
the triangulated category of D0-D2-D6 bound states
on a smooth projective Calabi-Yau 3-fold. 
\subsection{Weak stability conditions on triangulated 
categories}
In this subsection, we recall the 
notion of weak stability conditions on 
triangulated categories, and collect 
some results we need in the latter subsections.  
For the detail, see~\cite[Section~2]{Tcurve1}.
 
Let $\dD$ be a triangulated category, and $K(\dD)$
the Grothendieck group of $\dD$. 
We fix a finitely generated free abelian group $\Gamma$
together with a group homomorphism, 
\begin{align*}
\cl \colon K(\dD) \to \Gamma. 
\end{align*}
We also fix a filtration, 
\begin{align*}
\Gamma_{0} \subset \Gamma_{1} \subset \cdots 
\subset \Gamma_{N}=\Gamma, 
\end{align*}
such that each subquotient $\Gamma_{i}/\Gamma_{i-1}$
is a free abelian group. 
\begin{defi}\label{defi:weak}\emph{
A \textit{weak stability condition} on $\dD$ consists of 
data $(Z=\{Z_i\}_{i=0}^{N}, \aA)$, 
\begin{align*}
Z_i \colon \Gamma_{i}/\Gamma_{i-1} \to \mathbb{C}, \quad 
\aA \subset \dD, 
\end{align*}
where $Z_i$ are group homomorphisms and 
$\aA$ is the heart of a bounded t-structure 
on $\dD$, which satisfy the following. }
\begin{itemize}
\item \emph{For any non-zero $E\in \aA$
with $\cl(E) \in \Gamma_{i} \setminus \Gamma_{i-1}$, 
we have 
\begin{align}\label{cond}
Z(E)\cneq Z_i([\cl(E)]) \in \mathbb{H}. 
\end{align}
Here
$[\cl(E)] \in \Gamma_i/\Gamma_{i-1}$ is the class of 
$\cl(E) \in \Gamma_i \setminus \Gamma_{i-1}$ and 
 \begin{align*}
\mathbb{H}=\{r\exp(i\pi \phi) : r>0, 0<\phi \le 1\}.
\end{align*}
We say $E\in \aA$ is \textit{$Z$-(semi)stable} if for any 
exact sequence $0 \to F \to E \to G \to 0$
in $\aA$, 
we have} 
\begin{align*}
\arg Z(F) < (\le) \arg Z(G). 
\end{align*} 
\item  \emph{For any $E\in \aA$, there is a filtration
in $\aA$, (Harder-Narasimhan filtration,)
\begin{align*}
0=E_0 \subset E_1 \subset \cdots \subset E_n=E, 
\end{align*}
such that each subquotient $F_i=E_i/E_{i-1}$
is $Z$-semistable with
\begin{align*}
\arg Z(F_i)> \arg Z(F_{i+1}),
\end{align*}
for all $i$.} 
\end{itemize}
\end{defi}
Here we remark that for $N=0$, the pair 
$(Z, \aA)$ determines a stability condition by
Bridgeland~\cite{Brs1}. 

Let $(Z, \aA)$ be a weak stability condition on $\dD$. 
For $0<\phi \le 1$,
 the subcategory 
$\pP(\phi)\subset \dD$ is defined
 to be the category of $Z$-semistable
objects $E\in \aA$
satisfying 
\begin{align}\label{exp}
Z(E) \in \mathbb{R}_{>0}\exp(i\pi \phi). 
\end{align}
For other $\phi \in \mathbb{R}$, the subcategory 
$\pP(\phi)$ is determined by the rule, 
\begin{align*}
\pP(\phi+1)=\pP(\phi)[1]. 
\end{align*}
The family of subcategories 
$\pP(\phi)$ for $\phi \in \mathbb{R}$
determines a \textit{slicing} introduced in~\cite[Definition~3.3]{Brs1}. 
As in~\cite[Proposition~2.13]{Tcurve1}, 
giving a weak stability condition is equivalent 
to giving a data, 
\begin{align}\label{pasli}
\sigma=(Z=\{Z_i\}_{i=0}^N, \pP),
\end{align}
where $Z$ is as above and $\pP$ is a slicing, 
satisfying the condition (\ref{exp})
for any non-zero $E\in \pP(\phi)$.
In what follows, we occasionally 
write a weak stability 
condition as a pair of group homomorphisms 
$\{Z_i\}_{i=0}^{N}$ and 
a slicing $\pP$, as in (\ref{pasli}). 
The subcategory $\pP(\phi) \subset \dD$
is called the category of \textit{$\sigma$-semistable objects
of phase $\phi$.} 
The category $\pP(\phi)$ is easily 
seen to be an abelian category, and
we denote by $\pP_s(\phi) \subset \pP(\phi)$
the subcategory of simple objects. 
An object in $\pP_s(\phi)$ is called 
 a \textit{$\sigma$-stable object of phase $\phi$. }

 For an interval 
$I\subset \mathbb{R}$, we set 
\begin{align*}
\pP(I)\cneq \langle \pP(\phi) : \phi \in I \rangle_{\ex}.
\end{align*}
We also need the following technical conditions. 
\begin{itemize}
\item {\bf (Support property):}
There is a constant $C>0$ such that 
for any $E\in \pP(\phi)$
with $\cl(E) \in \Gamma_i \setminus \Gamma_{i-1}$, we have 
\begin{align*}
\lVert [\cl(E)] \rVert_{i} \le C\cdot \lvert Z(E) \rvert. 
\end{align*}
Here $\lVert \ast \rVert_{i}$ is a fixed 
norm on $(\Gamma_i/\Gamma_{i-1}) \otimes_{\mathbb{Z}} \mathbb{R}$. 
\item {\bf (Local finiteness):} 
There is $\varepsilon>0$ such that the quasi-abelian category 
$\pP((\phi-\varepsilon, \phi+\varepsilon))$ is of finite length
for any $\phi \in \mathbb{R}$. 
\end{itemize}
Here we refer~\cite[Definition~4.1, Definition~5.7]{Brs1}
for the detail on the notion of quasi-abelian categories and 
their finite length property. 
The set of weak stability conditions satisfying 
the above two properties is denoted by 
$\Stab_{\Gamma_{\bullet}}(\dD)$. 
The following result is an analogue
of~\cite[Theorem~7.1]{Brs1} and proved in~\cite[Theorem~2.15]{Tcurve1}. 
\begin{thm}\label{thm:top}
There is a natural topology on $\Stab_{\Gamma_{\bullet}}(\dD)$
such that the forgetting map 
\begin{align*}
\Pi \colon 
\Stab_{\Gamma_{\bullet}}(\dD) \ni (Z, \aA) \mapsto 
Z \in \prod_{i=0}^{N}\Hom_{\mathbb{Z}}(\Gamma_{i}/\Gamma_{i-1}, \mathbb{C}),
\end{align*}
is a local homeomorphism. In particular each connected component 
of $\Stab_{\Gamma_{\bullet}}(\dD)$ is a complex manifold. 
\end{thm}
\begin{rmk}\label{rmk216}
As mentioned in~\cite[Remark~2.16]{Tcurve1}, 
the set of $\sigma \in \Stab_{\Gamma_{\bullet}}(\dD)$
in which a fixed object $E\in \dD$ is $\sigma$-semistable 
is a closed subset. 
\end{rmk}
There is a continuous $\mathbb{C}$-action on the 
space $\Stab_{\Gamma_{\bullet}}(\dD)$ in
the following way. For a pair $\sigma=(Z, \pP)$
as in (\ref{pasli}) and $\lambda \in \mathbb{C}$, we 
set 
\begin{align*}
\lambda \cdot \sigma=(\exp(-i \pi \lambda)Z, \pP'), 
\end{align*}
where $\pP'$ is a slicing 
given by $\pP'(\phi)=\pP(\phi +\Ree \lambda)$
for all $\phi \in \mathbb{R}$.

For the heart of a bounded t-structure $\aA\subset \dD$, 
we denote by 
\begin{align*}
\hH_{\aA}^{i} \colon \dD \to \aA,
\end{align*} 
the $i$-th cohomology functor with respect to the 
t-structure with heart $\aA$. 
We will use the following notions of torsion 
pair and tilting to construct weak stability conditions. 
\begin{defi}\emph{
Let $\aA$ be the heart of a bounded t-structure 
on a triangulated category $\dD$. A pair of 
subcategories $(\tT, \fF)$ in $\aA$ is called 
a \textit{torsion pair} if the following conditions hold. }
\begin{itemize}
\item \emph{For any $T\in \tT$ and $F\in \fF$, we have
$\Hom(T, F)=0$.}
\item \emph{For any $E\in \aA$, there is an 
exact sequence 
\begin{align*}
0 \to T \to E \to F \to 0,
\end{align*}
for $T\in \tT$ and $F\in \fF$.}
\end{itemize}
\end{defi}
Given a torsion pair $(\tT, \fF)$ as above, its 
\textit{tilting} is defined by 
\begin{align*}
\aA^{\dag}&\cneq \left\{E \in \dD : \begin{array}{l}
 \hH_{\aA}^{0}(E) \in \fF, \ 
\hH_{\aA}^1(E) \in \tT,  \\
\hH^{i}(E)=0 \mbox{ for all }i\neq 0, 1. 
\end{array}
 \right\}, \\
&=\langle \fF, \tT[-1] \rangle_{\ex}.
\end{align*}
The category $\aA^{\dag}$ is also the 
heart of a bounded t-structure on $\dD$.
(cf.~\cite[Proposition~2.1]{HRS}.) 
\subsection{Construction of weak stability conditions}
Let $X$ be a smooth 
projective Calabi-Yau 3-fold satisfying 
\begin{align}\label{assum:van}
H^1(X, \oO_X)=0. 
\end{align}
We define the triangulated category $\dD_X$ to be 
\begin{align*}
\dD_X \cneq \langle \oO_X, \Coh_{\le 1}(X) \rangle_{\tr}
\subset D^b(\Coh(X)).
\end{align*}
We set the finitely generated abelian group $\Gamma$ to  
be 
\begin{align*}
\Gamma\cneq H^0(X, \mathbb{Z}) \oplus H_2(X, \mathbb{Z}) \oplus 
H_{0}(X, \mathbb{Z}). 
\end{align*}
By the Poincar\'e duality, 
the Chern characters of $E$ define a 
group homomorphism
$\cl \colon K(\dD_X) \to \Gamma$,
\begin{align*}
\cl(E)=(\ch_0(E), \ch_2(E), \ch_3(E)). 
\end{align*}
We set the two step filtration of $\Gamma$ to be 
\begin{align*}
\Gamma_{0}\cneq H^0(X, \mathbb{Z}) \subset \Gamma_{1}\cneq \Gamma. 
\end{align*}
We are going to study the space of weak stability conditions
$\Stab_{\Gamma_{\bullet}}(\dD_X)$. 
Note that 
\begin{align*}
\Hom(\Gamma_{0}, \mathbb{C}) &\cong \mathbb{C}, \\
\Hom(\Gamma_{1}/\Gamma_0, \mathbb{C})
&\cong H^2(X, \mathbb{C}) \oplus \mathbb{C}.
\end{align*}
The forgetting map $(Z, \aA) \mapsto Z$ is 
as follows, 
\begin{align}\label{forget}
\Pi \colon 
\Stab_{\Gamma_{\bullet}}(\dD_X) \to \mathbb{C}\times 
H^2(X, \mathbb{C}) \times \mathbb{C}. 
\end{align}
\begin{rmk}
As mentioned in~\cite[Remark~]{Tcurve1}, 
a weak stability condition in this situation
 may be interpreted to be 
a limiting point $m\to \infty$ of some sequence 
of stability conditions,  
\begin{align*}
\sigma^{(m)}=(Z^{(m)}, \cC^{(m)}),
\end{align*}
where $\cC^{(m)} \subset \dD_X$ is the heart of a bounded
t-structure and $Z^{(m)}\colon \Gamma \to \mathbb{C}$
is written as 
\begin{align*}
Z^{(m)}(r, \beta, n)=Z_0(r) +mZ_1(\beta, n). 
\end{align*}
Here $Z_i \colon \Gamma_i/\Gamma_{i-1} \to \mathbb{C}$
are group homomorphisms for $i=0, 1$. 
Note that we have 
\begin{align*}
\lvert Z^{(m)}(\cl(\oO_X))\rvert \ll 
\lvert Z^{(m)}(\cl(F)) \rvert, \quad 
m\gg 0,
\end{align*}
where $F\in \dD_X$ satisfies 
$\cl(F) \in \Gamma_1 \setminus \Gamma_0$.
This implies that the mass of 
the object $\oO_X$ is infinitesimally small 
w.r.t. our weak stability conditions. 
\end{rmk}
Here we construct three types of weak stability conditions
on $\dD_X$. 
\begin{lem}\label{types}
(i) There is the heart of a bounded t-structure
 $\aA\subset\dD_X$, written as 
\begin{align*}
\aA=\langle \oO_X, \Coh_{\le 1}(X)[-1] \rangle_{\ex}. 
\end{align*}

(ii) There is the heart of a bounded t-structure 
$\bB_{+} \subset \dD_X$, written as 
\begin{align*}
\bB_{+}=\langle \aA_{+}, \oO_X[-1] \rangle_{\ex}. 
\end{align*}
Here $\aA_{+}=\{E\in \aA : \Hom(\oO_X, E)=0\}$. 

(iii) There is the heart of a bounded t-structure 
$\bB_{-} \subset \dD_X$, written
as \begin{align*}
\bB_{-}=\langle \oO_X[1], \aA_{-} \rangle_{\ex}.
\end{align*}
Here $\aA_{-}=\{E\in \aA : \Hom(E, \oO_X)=0\}$. 
\end{lem}
\begin{proof}
The proof of (i) is given in~\cite[Lemma~3.5]{Tcurve1}. 
For the proof of (ii), note that the pair 
\begin{align*}
\langle \langle \oO_X \rangle_{\ex}, \aA_{+} \rangle_{\ex},
\end{align*}
is a torsion pair. This is easily checked 
by the fact that $\aA$ is a noetherian abelian 
category~\cite[Lemma~6.2]{Tcurve1}. 
The tilting with respect to the above torsion 
pair yields the heart $\bB_{+}$. 
The proof of (iii) is similar. 
\end{proof}
For a given data,
\begin{align*}
u=(z, B+i\omega) \in \mathbb{C}\times H^2(X, \mathbb{C}), 
\end{align*}
we associate the element,
\begin{align}\label{asso}
Z_{u}=\{Z_{u, i}\}_{i=0}^{1} \in \prod_{i=0}^{1}
\Hom_{\mathbb{Z}}(\Gamma_i/\Gamma_{i-1}, \mathbb{C}),
\end{align}
 as follows,
\begin{align}\label{Z1}
Z_{u, 0} &\colon \Gamma_{0}=H^0(X, \mathbb{Z}) \ni r \mapsto rz, \\
\notag
Z_{u, 1} &\colon \Gamma_{1}/\Gamma_{0} =H_2(X, \mathbb{Z}) \oplus
 H_0(X, \mathbb{Z}) \\ 
\label{Z2}
& \qquad \qquad 
\ni (\beta, n) \mapsto n-(B+i\omega)\cdot \beta. 
\end{align}
Let $A(X)_{\mathbb{C}} \subset H^2(X, \mathbb{C})$ be
the complexified ample cone, 
\begin{align*}
A(X)_{\mathbb{C}}=\{B+i\omega \in H^2(X, \mathbb{C}) : \omega
\mbox{ is an ample }\mathbb{R} \mbox{ divisor. } \}. 
\end{align*}
We have the following lemma. 
\begin{lem}\label{lem:cons}
(i) The pairs  
\begin{align}\label{pairA}
\sigma_u=(Z_{u}, \aA), \quad 
 u\in \mathbb{H}\times A(X)_{\mathbb{C}},
\end{align}
determine points in $\Stab_{\Gamma_{\bullet}}(\dD_X)$. 

(ii) The pairs 
\begin{align*}
\tau_{u\pm}=(Z_{u}, \bB_{\pm}), \quad 
 u\in (-\mathbb{H})\times A(X)_{\mathbb{C}}, 
\end{align*}
determine points in $\Stab_{\Gamma_{\bullet}}(\dD_X)$. 
\end{lem}
\begin{proof}
The proofs of some technical conditions (Harder-Narasimhan 
property, support property, local finiteness) are
postponed until Section~\ref{sec:tech}. 
Here we only check that the condition (\ref{cond}) 
is satisfied. 

(i)  For $E\in \aA$, 
let us write 
\begin{align*}
\cl(E)=(r, -\beta, -n) \in H^0 \oplus H_2 \oplus H_0. 
\end{align*}
Suppose that $\cl(E) \in \Gamma_{1} \setminus \Gamma_{0}$. 
Then the description (\ref{pairA}) shows that 
$\beta$
is an effective curve class
 or $\beta=0, n>0$. 
Hence $Z_{u}(E) \in \mathbb{H}$ follows. 
If $\cl(E) \in \Gamma_{0}$, then 
$E\in \langle \oO_X \rangle_{\ex}$, hence 
we have 
$Z_{u}(E)=rz \in \mathbb{H}$. 

(ii) For simplicity we check the case of 
$(Z_{u}, \bB_{+})$. 
For an object 
$E\in \bB_{+}$, there is an exact 
sequence in $\bB_{+}$, 
\begin{align*}
0 \to T \to E \to F \to 0, 
\end{align*}
with $T\in \aA_{+}$ and $F \in \langle \oO_X[-1] \rangle_{\ex}$. 
If $\cl(E) \in \Gamma_{1}\setminus \Gamma_{0}$, 
then $T \neq 0$ and we have 
\begin{align*}
Z_u(E)=Z_{u}(T) \in \mathbb{H}, 
\end{align*}
by
the same argument of (i). 
If $\cl(E) \in \Gamma_0$, then 
 we have $E \in \langle \oO_X[-1] \rangle_{\ex}$,
hence 
we have $Z_{u}(E)= rz \in \mathbb{H}$.  
\end{proof}
\subsection{Standard regions in the space $\Stab_{\Gamma_{\bullet}}(\dD_X)$.}
The constructions of weak stability conditions in the 
last subsection yield some standard regions in the 
space $\Stab_{\Gamma_{\bullet}}(\dD_X)$. 
We set $\uU$ and $\uU_{\pm 1}$ to be
\begin{align}\label{region1}
\uU &\cneq \{\sigma_{u} \in \Stab_{\Gamma_{\bullet}}(\dD_X) : 
u\in \mathbb{H}\times A(X)_{\mathbb{C}}\}, \\
\label{region2}
\uU_{\pm 1} &\cneq \{\tau_{u\pm} \in \Stab_{\Gamma_{\bullet}}(\dD_X) : 
u\in (-\mathbb{H})\times A(X)_{\mathbb{C}}\}.
\end{align}
The above subspaces lie in the 
 space of normalized weak 
stability conditions,
\begin{align*}
\Stab_{\Gamma_{\bullet}, \rm{n}}(\dD_X)
\cneq \{(Z, \aA) \in \Stab_{\Gamma_{\bullet}}(\dD_X):
Z(\oO_x)=-1\},
\end{align*}
where $x\in X$ is a closed point. 
The forgetting map (\ref{forget}) restricts to the local 
homeomorphism, 
\begin{align*}
\Pi_{\rm n}\colon 
\Stab_{\Gamma_{\bullet}, \rm{n}}(\dD_X)
\to \mathbb{C}\times H^2(X, \mathbb{C}). 
\end{align*}
We have the following lemma. 
\begin{lem}\label{conne}
(i) 
The map $\Pi_{\rm n}$ 
restrict to the homeomorphisms, 
\begin{align*}
\Pi_{\rm n} &\colon \uU \stackrel{\sim}{\to}
\mathbb{H}\times A(X)_{\mathbb{C}}, \\
\Pi_{\rm n} &\colon 
\uU_{\pm 1} \stackrel{\sim}{\to}
(-\mathbb{H})\times A(X)_{\mathbb{C}}. 
\end{align*}
(ii) 
The map $\Pi_{\rm n}$
restricts to the homeomorphisms, 
\begin{align*}
\Pi_{\rm n} \colon \uU \cap \overline{\uU}_{+1}
\stackrel{\sim}{\to} \mathbb{R}_{<0} \times A(X)_{\mathbb{C}}, \\
\Pi_{\rm n} \colon \overline{\uU} \cap \uU_{-1}
\stackrel{\sim}{\to} \mathbb{R}_{>0} \times A(X)_{\mathbb{C}}.
\end{align*}
\end{lem}
\begin{proof}
Since $\bB_{\pm}$ is obtained from $\aA$ 
by tilting, both of (i) and (ii) follow by applying
Lemma~\ref{cC} below.  
\end{proof}
We have used the following lemma, 
whose proof is given in~\cite[Lemma~7.1]{Tcurve1}. 
\begin{lem}\emph{\bf{\cite[Lemma~7.1]{Tcurve1}}}
\label{cC}
Let $\cC$ be the heart of a bounded t-structure on 
$\dD_X$ and $(\tT, \fF)$ a torsion pair on $\cC$. 
Let $\cC'=\langle \fF, \tT[-1]\rangle_{\ex}$
be the associated tilting. Let 
\begin{align*}
[0, 1) \ni t \mapsto Z_t \in \prod_{i=0}^{1} \Hom_{\mathbb{Z}}
(\Gamma_i/\Gamma_{i-1}, \mathbb{C}),
\end{align*}
be a continuous map such that $\sigma_t=(Z_t, \cC)$
for $0<t<1$ and $\sigma_0=(Z_{0}, \cC')$
determine points in $\Stab_{\Gamma_{\bullet}}(\dD_X)$. 
Then we have $\lim_{t\to 0}\sigma_t =\sigma_0$. 
\end{lem}

By Lemma~\ref{conne}, 
the subspaces $\uU$, $\uU_{\pm 1}$ are 
contained in the same connected component, 
which we denote by 
\begin{align*}
\Stab_{\Gamma_{\bullet}, \rm{n}}^{\circ}(\dD_X)
\subset \Stab_{\Gamma_{\bullet}, \rm{n}}(\dD_X). 
\end{align*}
\subsection{Weak stability conditions and 
Seidel-Thomas twist}
By our assumption (\ref{assum:van}), the object 
$\oO_X$ is a \textit{spherical object}, i.e. 
\begin{align*}
\Ext_{X}^{i}(\oO_X, \oO_X)=\left\{ 
\begin{array}{cc}
\mathbb{C}, & i=0, 3, \\
0, & \mbox{ otherwise. }
\end{array} \right. 
\end{align*} 
We have the associated derived equivalence, 
called \textit{Seidel-Thomas twist}~\cite{ST},
\begin{align}\label{equ:Se}
\Phi_{\oO_X} \colon D^b(\Coh(X)) \stackrel{\sim}{\to}
D^b(\Coh(X)). 
\end{align}
The above equivalence has the property that 
there is a 
 distinguished triangle, 
\begin{align}\label{dist}
\dR \Hom(\oO_X, E) \otimes_{\mathbb{C}}\oO_X 
\to E \to \Phi_{\oO_X}(E),
\end{align}
for any object $E\in D^b(\Coh(X))$. 
\begin{lem}\label{com2}
The equivalence $\Phi_{\oO_X}$ preserves 
the subcategory $\dD_X$, and we have the 
commutative diagram, 
\begin{align}\label{com:sei}
\xymatrix{
K(\dD_X) \ar[r]^{\Phi_{\oO_X}} \ar[d]_{\cl} & 
K(\dD_X) \ar[d]^{\cl} \\
\Gamma \ar[r]^{\phi_{\oO_X}}& \Gamma.
}
\end{align}
Here $\phi_{\oO_X}$ is given by 
\begin{align*}
\phi_{\oO_X}(r, \beta, n)=(r-n, \beta, n), 
\end{align*}
for $(r, \beta, n)\in H^0 \oplus H_2 \oplus H_0$.
\end{lem}
\begin{proof}
By the distinguished triangle (\ref{dist}), 
it is obvious that the equivalence $\Phi_{\oO_X}$
preserves $\dD_X$. 
Since $\ch_1(E)=0$ for any $E\in \dD_X$, 
the Riemann-Roch theorem yields, 
\begin{align*}
\sum_{i}(-1)^i \dim \Hom(\oO_X, E[i])=\ch_3(E). 
\end{align*}
Then the distinguished triangle (\ref{dist}) 
implies that the diagram (\ref{com:sei}) is commutative. 
\end{proof}
Note that $\phi_{\oO_X}$
preserves the filtration $\Gamma_{\bullet}$
and the induced map on $\gr (\Gamma_{\bullet})$ is identity. 
Hence by Lemma~\ref{com2} and~\cite[Lemma~2.9]{Tcurve2},
we have the commutative diagram, 
\begin{align*}
\xymatrix{
\Stab_{\Gamma_{\bullet}}(\dD_X)\ar[r]^{\Phi_{\oO_X \ast}} \ar[d]_{\Pi}
& \Stab_{\Gamma_{\bullet}}(\dD_X) \ar[d]^{\Pi} \\
\gr(\Gamma_{\bullet})^{\vee}\otimes \mathbb{C} \ar[r]^{\id} & 
\gr (\Gamma_{\bullet})^{\vee}\otimes \mathbb{C}. }
\end{align*} 
Here $\Phi_{\oO_X \ast}$ is given by,
\begin{align*}
\Phi_{\oO_X \ast}(Z, \aA)=(Z, \Phi_{\oO_X}(\aA)),
\end{align*}
where $\aA \subset \dD_X$ is the heart of a bounded t-structure.
It is obvious that $\Phi_{\oO_X \ast}$ preserves
the normalized weak stability conditions, 
so there is a commutative diagram,
\begin{align*}
\xymatrix{
\Stab_{\Gamma_{\bullet}, \rm{n}}(\dD_X) \ar[d]_{\Pi_{\rm n}} 
 \ar[r]^{\Phi_{\oO_X \ast}} & 
\ar[d]^{\Pi_{\rm n}} \Stab_{\Gamma_{\bullet}, \rm{n}}(\dD_X) \\
 \mathbb{C}\times H^2(X)_{\mathbb{C}} \ar[r]^{\id} &
\mathbb{C}\times H^2(X)_{\mathbb{C}}.}
\end{align*}
Under the Seidel-Thomas twist (\ref{equ:Se}), 
the regions (\ref{region2}) are related as follows. 
\begin{lem}\label{phist}
We have 
\begin{align*}
\Phi_{\oO_X \ast}\uU_{-1}=\uU_{+1}. 
\end{align*}
In particular, $\Phi_{\oO_X \ast}$ preserves
the connected component $\Stab_{\Gamma_{\bullet}, \rm{n}}^{\circ}(\dD_X)$. 
\end{lem}
\begin{proof}
By the construction of $\uU_{\pm 1}$, it 
is enough to show that 
\begin{align*}
\Phi_{\oO_X}(\bB_{-})=\bB_{+}. 
\end{align*}
Since both sides are hearts of bounded t-structures, 
it is enough to see that the LHS is contained in the RHS.
By Lemma~\ref{types} (ii), this follows by showing that 
\begin{align*}
\Phi_{\oO_X}(\oO_X[1]) \in \bB_{+}, \quad 
\Phi_{\oO_X}(\aA_{-}) \subset \bB_{+}. 
\end{align*}
First 
it is easy to see that
\begin{align}\label{PhiO}
\Phi_{\oO_X}(\oO_X[1])=\oO_X[-1] \in \bB_{+},
\end{align}
using the distinguished triangle (\ref{dist}). 
Next let us take $E\in \aA_{-}$, and 
show that $\Phi_{\oO_X}(E) \in \bB_{+}$. We
set 
\begin{align*}
r_i=\dim \Hom(\oO_X, E[i]). 
\end{align*}
The Serre duality 
implies that $r_3=0$. 
Applying $\hH_{\aA}^{\bullet}$ to the 
distinguished triangle (\ref{dist})
and noting that $\oO_X \in \aA$ is a simple object, 
it is easy to see that  
\begin{align*}
\hH_{\aA}^{i}(\Phi_{\oO_X}(E))=0, \ i\neq 0, 1, 
\quad 
\hH_{\aA}^{1}(\Phi_{\oO_X}(E))\cong \oO_X^{\oplus r_2}. 
\end{align*}
Also this implies that 
\begin{align*}
\Hom(\oO_X, \hH_{\aA}^{0}(\Phi_{\oO_X}(E)))
&\cong \Hom(\oO_X, \Phi_{\oO_X}(E)) \\
&\cong \Hom(\oO_X[2], E) \\
&\cong 0. 
\end{align*}
Therefore $\hH_{\aA}^{0}(\Phi_{\oO_X}(E)) \in \aA_{+}$.  
By the construction of $\bB_{+}$, we conclude that   
$\Phi_{\oO_X}(E) \in \bB_{+}$.
\end{proof}
Applying the twist functor $\Phi_{\oO_X}$
to the regions (\ref{region1}), (\ref{region2}), 
we can construct other regions in 
the space $\Stab_{\Gamma_{\bullet}, \rm{n}}^{\circ}(\dD_X)$. 
For $k\in \mathbb{Z}$, they are defined by the following way, 
\begin{align}\label{stand1}
\uU_{2k} &\cneq \Phi_{\oO_X \ast}^{(k)}(\uU), \\
\label{stand2}
\uU_{2k+1} &\cneq \Phi_{\oO_X \ast}^{(k)}(\uU_{+1}). 
\end{align}
We have the following lemma. 
\begin{lem}\label{Ost}
For $\sigma=(Z, \pP) \in \uU_k$, 
we have 
\begin{align*}
\oO_X \in \pP_{s}(\phi), \quad k < \phi \le k+1. 
\end{align*}
\end{lem}
\begin{proof}
It is easy to check that the 
objects 
\begin{align*}
\oO_X \in \aA, \quad \oO_X[-1] \in \bB_{+},
\end{align*}
are simple objects in $\aA$, $\bB_{+}$ respectively. 
Therefore the statement 
follows for $k=0, 1$. 
Since $\Phi_{\oO_X}(\oO_X)=\oO_X[-2]$, 
the result also follows for all $k\in \mathbb{Z}$. 
\end{proof}

\subsection{Semistable sheaves and semistable objects}
In this subsection, we recall the 
classical notion of (semi)stability on the 
category $\Coh_{\le 1}(X)$, and 
compare it with our weak stability conditions. 
For $B+i\omega \in A(X)_{\mathbb{C}}$ and 
$F\in \Coh_{\le 1}(X)$, we set
\begin{align*}
\mu_{(B, \omega)}(F)=\frac{\ch_3(E)+B\cdot \ch_2(E)}{\omega \cdot \ch_2(E)}.
\end{align*}
\begin{defi}\emph{
We say $F$ is a \textit{$(B, \omega)$-(semi)stable sheaf} 
if for any non-zero proper subsheaf $F'\subset F$, we
have 
\begin{align*}
\mu_{(B, \omega)}(F')<(\le) \mu_{(B, \omega)}(F).
\end{align*}}
\end{defi}
If $B=0$, we simply write $\mu_{(0, \omega)}(\ast)=\mu_{\omega}(\ast)$
and call a $(0, \omega)$-(semi)stable sheaf
just an $\omega$-(semi)stable sheaf. 
Note that we have the inclusions, 
\begin{align*}
\Coh_{\le 1}(X)[-1] \subset \aA, \ \bB_{\pm}, 
\end{align*}
hence it is natural to relate $(B, \omega)$-stability 
with our weak stability conditions. 
The following lemma will not be 
needed except in showing Lemma~\ref{lem:dense} below, 
but it helps us to see what kind of objects 
appear as semistable objects w.r.t. our 
weak stability conditions. 
The proof will be given in Section~\ref{sec:tech}. 
\begin{lem}\label{omit}
(i) Take $u=(z, B+i\omega) \in \mathbb{H} \times A(X)_{\mathbb{C}}$
and a $(B, \omega)$-semistable sheaf $F \in \Coh_{\le 1}(X)$ 
Then the object 
$F[-1] \in \aA$ is a $Z_{u}$-semistable
object. 

(ii) 
Take $u=(z, B+i\omega) \in (-\mathbb{H})\times A(X)_{\mathbb{C}}$
and a $(B, \omega)$-semistable 
sheaf $F\in \Coh_{\le 1}(X)$.
Then we have the following. 
\begin{itemize}
\item Suppose that $\arg Z_u(F[-1]) > \arg (-z)$. 
Then the object $F[-1] \in \bB_{+}$ is $Z_u$-semistable. 
\item Suppose that $\arg Z_u(F[-1])< \arg (-z)$. 
Then we have $\Phi_{\oO_X}(F[-1]) \in \bB_{+}$ and it is 
$Z_u$-semistable. 
\end{itemize}
\end{lem}

\subsection{The space of normalized 
weak stability conditions}
Let $\{\uU_i\}_{i\in \mathbb{Z}}$
be the family of regions 
constructed in (\ref{stand1}), (\ref{stand2}). 
We have the following 
description of the space of
 normalized weak stability conditions. 
\begin{thm}\label{thm:cov}
Assume that
\begin{align}\label{add}
H^2(X, \mathbb{Z}) \cong \mathbb{Z}.
\end{align}
Then we have the following,
\begin{align}\label{copro}
\Stab_{\Gamma_{\bullet}, \rm{n}}^{\circ}(\dD_X)=
\coprod_{i \in \mathbb{Z}}\uU_i.
\end{align}
In particular, the forgetting map is a 
universal covering map, 
\begin{align*}
\Pi_{\rm{n}}\colon \Stab_{\Gamma_{\bullet}, \rm{n}}^{\circ}(\dD_X)
\to \mathbb{C}^{\ast}\times \mathbb{H}^{\circ},
\end{align*}
with Galois group generated by $\Phi_{\oO_X \ast}$. 
\end{thm}
\begin{proof}
By Lemma~\ref{conne}, the RHS of (\ref{copro})
is an open subset in the LHS. Hence it is enough to show
that the RHS is closed in the LHS.  
By Lemma~\ref{Ost}, 
the RHS is a locally finite union, i.e. 
for any compact subset $\mathfrak{B} \subset 
\Stab_{\Gamma_{\bullet}, \rm{n}}^{\circ}(\dD_X)$, 
the number of $i\in \mathbb{Z}$
satisfying $\uU_{i} \cap \mathfrak{B} \neq \emptyset$
is finite. This
implies that 
\begin{align*}
\overline{\coprod_{i\in \mathbb{Z}}\uU_i}
=\coprod_{i\in \mathbb{Z}}\overline{\uU_i}. 
\end{align*}
Hence it is enough to show that $\overline{\uU_i}$
is contained in the RHS of (\ref{copro}). Furthermore 
applying $\Phi_{\oO_X}$, 
we may assume that $i=0$ or $i=1$. 
For simplicity, we show the case of $i=1$. 
The other case is similarly discussed. 

Let us take a point $\sigma \in \overline{\uU_1}\setminus \uU_1$. 
We can write $\sigma=(Z_u, \pP)$
for $u=(z, B+i\omega) \in (-\mathbb{H})\times A(X)_{\mathbb{C}}$
 and a slicing $\pP$. 
Since $\oO_X$ is stable in $\uU_{1}$, 
it is also semistable in $\sigma$, 
hence we have $z\neq 0$. 
By the assumption (\ref{add}), 
we have the following possibilities. 

(i) $z\in \mathbb{R}_{<0}, \quad \omega \neq 0.$ 

(ii) $z\in -\overline{\mathbb{H}}\setminus \{0\}, \quad
\omega=0.$ 

Suppose that (i) holds. 
Then $\sigma \in \uU_{0}$
by Lemma~\ref{conne}
and $\sigma$ is contained in the RHS of (\ref{copro}). 
We show that the case (ii) doesn't happen. 

Suppose by contradiction that (ii) holds. 
We set 
\begin{align*}
\phi_{0}=\frac{1}{\pi} \arg(-z) \in [0, 1].
\end{align*}
Since $\omega=0$, 
we have $\pP(\phi)=\{0\}$
unless $\phi \in \mathbb{Z}$ or 
$\phi \in \mathbb{Z} + \phi_{0}$. 
Therefore we can find $\psi \in (0, 1)$ 
and $0<\varepsilon \ll 1$ satisfying 
\begin{align*}
(\psi-2\varepsilon, \psi+2\varepsilon) \subset (0, 1)\setminus \{\phi_0\}. 
\end{align*}
Since $\sigma \in \overline{\uU_1}$, there is 
$\tau =(Z', \pP') \in \uU_1$
satisfying 
$\pP'(\phi) \subset \pP((\phi-\varepsilon, \phi+\varepsilon))$
for all $\phi \in \mathbb{R}$. Then we obtain 
\begin{align*}
\pP'((\psi-\varepsilon, \psi+\varepsilon))
\subset \pP((\psi-2\varepsilon, \psi+2\varepsilon))
=\{0\}. 
\end{align*}
However this contradicts to Lemma~\ref{lem:dense} below.  
The result for the forgetting map easily follows 
from (\ref{copro}), Lemma~\ref{conne} and Lemma~\ref{Ost}.  
\end{proof}
We have used the following lemma, 
whose proof will be given in Section~\ref{sec:tech}. 
\begin{lem}\label{lem:dense}
For $u \in (-\mathbb{H})\times A(X)_{\mathbb{C}}$,
we write $\tau_{u+}=(Z_{u}, \pP)$ for a slicing 
$\pP$. Then the set 
\begin{align}\label{dense}
\{ \phi \in \mathbb{R} : \pP(\phi) \neq \{0\} \} \subset \mathbb{R},
\end{align}
is dense in $\mathbb{R}$. 
\end{lem}
\subsection{The space of non-normalized 
weak stability conditions}
In this subsection, we investigate
the space of non-normalized weak stability 
conditions.
Let 
\begin{align*}
\Stab_{\Gamma_{\bullet}}^{\circ}(\dD_X) \subset 
\Stab_{\Gamma_{\bullet}}(\dD_X),
\end{align*}
be the connected component 
which contains $\uU$. 
 We show the following lemma. 
\begin{lem}\label{nonzero}
For any $\sigma=(Z, \pP) \in \Stab_{\Gamma_{\bullet}}^{\circ}(\dD_X)$,
we have $Z(\oO_x)\neq 0$
for any closed point $x\in X$. 
\end{lem}
\begin{proof}
Suppose by contradiction that $Z(\oO_x)=0$. 
We set the set of objects $\sS$ to be
\begin{align*}
\sS =\{ E\in \dD_X : \cl(E)=(0, 0, 1)\}. 
\end{align*}
By the condition $Z(\oO_x)=0$, 
there is no $a, b\in \mathbb{R}$
satisfying
\begin{align*}0<b-a \le 1, \quad 
\pP((a, b]) \cap \sS \neq \{0\}.
\end{align*}
By deforming $\sigma$, we can find 
 $\tau=(W, \qQ) \in \Stab_{\Gamma_{\bullet}}^{\circ}(\dD_X)$
such that $W(\oO_x)\neq 0$ and 
there is no
$a, b\in \mathbb{R}$
satisfying 
\begin{align*}
0<b-a \le 1/2, \quad \qQ((a, b]) \cap \sS \neq \{0\}. 
\end{align*}
This in particular implies that there is no 
 $\tau$-semistable 
object $E\in \sS$. 
Since $W(\oO_x)\neq 0$, we can 
apply $\mathbb{C}$-action on $\Stab_{\Gamma_{\bullet}}^{\circ}(\dD_X)$
to find,
\begin{align*}
\sigma'=(Z', \pP') \in \Stab_{\Gamma_{\bullet}, \rm{n}}^{\circ}(\dD_X), 
\end{align*}
such that there is no $\sigma'$-semistable 
object $E\in \sS$. 
However it is easy to check that 
the object
$\oO_x[-1]$
for a closed point $x\in X$ is
a simple object in both
$\aA$ and $\bB_{-}$, 
hence $\oO_x \in \sS$ is a stable object in 
$\uU$ and $\uU_{-1}$. 
Applying $\Phi_{\oO_X}$-action 
and Theorem~\ref{thm:cov}, 
we obtain a contradiction. 
\end{proof}
The relationship between 
normalized stability conditions and non-normalized
stability conditions is described as follows. 
\begin{prop}\label{prop:re}
The $\mathbb{C}$-action on 
$\Stab_{\Gamma_{\bullet}}^{\circ}(\dD_X)$
induces a commutative diagram,
\begin{align*}
\xymatrix{
\Stab_{\Gamma_{\bullet}, \rm{n}}^{\circ}(\dD_X)
\times \mathbb{C} \ar[r]^{\alpha} 
\ar[d]_{\Pi_{\rm n}\times \rm{id}}
& \Stab_{\Gamma_{\bullet}}^{\circ}(\dD_X)
\ar[d]^{\Pi} \\
\mathbb{C}^{\ast}\times \mathbb{H}^{\circ}
\times 
\mathbb{C}
\ar[r]^{e} & \mathbb{C}^3.}
\end{align*}
Here $\alpha$ is an isomorphism 
and $e$ is a map defined by 
\begin{align*}
e(s, t, u)=
(\exp(-i\pi u)s, \exp(-i\pi u)t, \exp(-i\pi u)).
\end{align*}
\end{prop}
\begin{proof}
The diagram is 
obviously commutative by the construction, 
so it is enough to show that $\alpha$
is an isomorphism. 
By Lemma~\ref{nonzero}, 
the map $\alpha$ is surjective, 
and it remains to check that $\alpha$ is injective. 
Take two elements, 
\begin{align*}
(\sigma_i, \lambda_i)
\in \Stab_{\Gamma_{\bullet}, \rm{n}}^{\circ}(\dD_X) \times \mathbb{C},
\quad i=1, 2, 
\end{align*}
which are mapped to the same element under $\alpha$. 
We may assume that $\lambda_1=0$, and $\sigma_1 \in \uU_{0} \cup \uU_{1}$
by Theorem~\ref{thm:cov}. For simplicity we show the case 
of $\sigma_1 \in \uU_{0}$. The other case is similarly 
discussed. 

Let us write
$\sigma_1=(Z_1, \aA)$ and $\sigma_2=(Z_2, \aA_2)$
for the heart of a bounded t-structure $\aA_2 \subset \dD_X$. 
Since $\lambda_1=0$
and $Z_i(\oO_x)=-1$
for $i=1, 2$, we obtain 
$\exp(-i \pi \lambda_2)=1$. 
Hence we may write $\lambda_2=2m$
for some $m\in \mathbb{Z}$. 
By Theorem~\ref{thm:cov} and 
Lemma~\ref{Ost} we 
can write the heart $\aA_2$ in 
two ways, 
\begin{align*}
\aA_2=\aA[2m]=\Phi_{\oO_X}^{(m)}(\aA). 
\end{align*}
Therefore 
we have the autoequivalence,
\begin{align}\label{aute}
(\Phi_{\oO_X}[-2])^{(m)}
\colon \aA \stackrel{\sim}{\to} 
\aA, 
\end{align}
which takes $\oO_X$ to $\oO_X$. 
Since the equivalence (\ref{aute})
takes simple objects to simple 
objects, it takes an object $\oO_{x}[-1]$
for $x\in X$
to an object
of the form $\oO_{x'}[-1]$
for some
$x' \in X$. 
Then a standard argument
(cf.~\cite[Theorem~2.5]{B-M1}) shows that 
\begin{align}\label{staut}
(\Phi_{\oO_X}[-2])^{(m)} \cong f^{\ast},
\end{align}
for an automorphism $f \colon X \stackrel{\sim}{\to} X$. However
by Lemma~\ref{com2},
the isomorphisms on $\Gamma$
induced by both sides
of (\ref{staut}) are equal only if 
$m=0$. Therefore $\lambda_2=0$ and $\sigma_1=\sigma_2$
follows.  
\end{proof}
Note that we have 
\begin{align*}
\Imm e=\mathbb{C}^{\ast} \times\GL_{+}(2, \mathbb{R}). 
\end{align*}
Here $\GL_{+}(2, \mathbb{R})$ 
is the subgroup of $\GL(2, \mathbb{R})$
preserving the orientation of $\mathbb{R}^2$, 
and it 
is embedded into $\mathbb{C}^2$
via 
\begin{align*}
\left( \begin{array}{cc}
a & b \\
c & d
\end{array} \right) \mapsto 
(a+ci, b+di). 
\end{align*}
Therefore we obtain the following theorem. 
\begin{thm}\label{obtain}
We have the isomorphism, 
\begin{align}\label{isom}
\Stab_{\Gamma_{\bullet}}^{\circ}(\dD_X)
\cong \mathbb{C}\times \widetilde{\GL}_{+}(2, \mathbb{R}),
\end{align}
and the isomorphism
of the double quotient space,  
\begin{align*}
\langle \Phi_{\oO_X} \rangle \backslash 
\Stab_{\Gamma_{\bullet}}^{\circ}(\dD_X) /
\mathbb{C} \cong  \mathbb{C}^{\ast} \times \mathbb{H}^{\circ}. 
\end{align*}
\end{thm}

\subsection{A loop around the conifold point}\label{sub:loop}
Let $S^1 \subset \mathbb{C}^{\ast}$
be the unit circle, 
and $\iota$ be the embedding, 
\begin{align*}
\iota=(\id, \sqrt{-1}) \colon 
S^1 \hookrightarrow \mathbb{C}^{\ast}
\times \mathbb{H}^{\circ}.
\end{align*} 
The embedding $\iota$ lifts to 
a map from the universal cover 
$\mathbb{R} \to S^1$, 
i.e. there is a commutative diagram, 
\begin{align}\label{iecom}
\xymatrix{
\mathbb{R} \ar[r]\ar[d]_{\exp(i\pi \ast)} & 
\Stab_{\Gamma_{\bullet}}^{\circ}(\dD_X)\ar[d], \\
S^1 \ar[r]^{\iota} & \mathbb{C}^{\ast} \times \mathbb{H}^{\circ}.
}
\end{align}
The top arrow of (\ref{iecom}) is denoted by $\gamma$, 
\begin{align}\label{map:g}
\gamma \colon 
\mathbb{R} \ni t \mapsto 
\gamma(t)=(Z_t, \pP_t) \in \Stab_{\Gamma_{\bullet}}^{\circ}(\dD_X),
\end{align}
where $\pP_t$ is a slicing of $\dD_X$
and the commutative diagram
(\ref{iecom}) implies 
\begin{align*}
Z_t =Z_{(\exp(i \pi t), i\omega)},
\end{align*}
for an ample generator $\omega
\in H^2(X, \mathbb{Z})$. 
Here the RHS is defined 
by (\ref{asso}) with $u=(\exp(i\pi t), i\omega)$. 
The map $\gamma$ is uniquely determined by 
requiring that 
\begin{align*}
\gamma((k, k+1]) \subset \uU_{k}, 
\quad \mbox{ for all }k \in \mathbb{Z}.
\end{align*}
Namely
we have $\pP_t((0, 1])=\aA_k$ for 
$t\in (k, k+1]$
with $k\in \mathbb{Z}$, 
where $\aA_k$ are hearts of bounded t-structures given by
\begin{align}\label{def:Ak}
\aA_{k} \cneq \left\{ \begin{array}{ll}
\Phi_{\oO_X}^{(k')}(\aA), & k=2k', \\
\Phi_{\oO_X}^{(k')}(\bB_{+1}), & k=2k'+1.
\end{array}  \right.
\end{align}
In what follows, we fix the continuous map (\ref{map:g}). 
We will use the following lemma.  
\begin{lem}\label{below}
For $E\in \aA_{k}$, we have 
$\cl(E)\in \Gamma_0$ if and only if 
$E\in \langle \oO_X[-k] \rangle_{\ex}$. 
In this case $E$ is a $Z_t$-semistable 
object of phase $t$. 
\end{lem}
\begin{proof}
Since $\Phi_{\oO_X}$ preserves the filtration $\Gamma_{\bullet}$
and we have (\ref{PhiO}), 
we may assume that $k=0$ or $k=1$.
In both cases, the assertion is easily checked by 
the construction of weak stability conditions
and Lemma~\ref{Ost}. 
\end{proof}

For $t\in \mathbb{R}$, 
the slicing $\pP_t$ 
defined in (\ref{map:g}) satisfies the following. 
\begin{lem}\label{fixed}
For fixed $\phi \in \mathbb{R}$ 
and $k \in \mathbb{Z}$, 
the subcategory $\pP_t(\phi) \subset \dD_X$
does not depend on $t\in (\phi+k, \phi+k+1)$. 
\end{lem}
\begin{proof}
Let us take 
\begin{align*}\phi+k<t_1<t_2<\phi+k+1,
\end{align*} 
and show that
$\pP_{t_1}(\phi) \subset \pP_{t_2}(\phi)$. 
The other inclusion $\pP_{t_2}(\phi) \subset \pP_{t_1}(\phi)$
is similarly discussed. 
We take an object $E\in \pP_{t_1}(\phi)$ 
and set 
\begin{align*}
I=\{ t\in [t_1, t_2] : E \in \pP_t(\phi)\}. 
\end{align*}
Since $I$ is a closed subset, 
(cf.~Remark~\ref{rmk216},) it
is enough to see that 
$t_0 \cneq \sup I$
satisfies $t_0=t_2$. 
Suppose by contradiction that $t_0<t_2$. 
Then there is a distinguished triangle, 
\begin{align}\label{disttr}
E' \to E \to E'',
\end{align}
which destabilizes $E$ 
w.r.t. the weak stability condition $\gamma(t)$
for $0<t-t_0 \ll 1$. 
If both of $\cl(E')$ and $\cl(E'')$
are not contained in $\Gamma_0$, we have 
\begin{align*}
\arg Z_{t_0}(E')=\arg Z_{t}(E')> \arg Z_{t}(E'') =\arg Z_{t_0}(E''). 
\end{align*}
This contradicts to that $E\in \pP_{t_0}(\phi)$, therefore 
either $\cl(E')$ or $\cl(E'')$ is contained 
in $\Gamma_0$. 
Then by Lemma~\ref{below}, 
there is $k' \in \mathbb{Z}$
such that $E_1$ or $E_2$ is contained in 
$\langle \oO_X[k'] \rangle_{\ex}$. This implies that  
\begin{align*}
t+k'>\phi\ge t_0 +k' \ \mbox{ or } \ t_0 +k' \ge \phi >t+k', 
\end{align*}
respectively. Obviously both 
cases do not happen. 
\end{proof}
Let us take $0<\phi <1$, 
$k\in \mathbb{Z}$ and set $t_{0}=\phi+k$. 
We take $t_{-}<t_0<t_{+}$ satisfying 
\begin{align}\label{tsat}
[t_{-}, t_{+}] \subset (k, k+1). 
\end{align}
Note that $\gamma(t) \in \uU_k$
for $t\in [t_{-}, t_{+}]$. 
We have the following proposition. 
\begin{prop}\label{prop1}
For an object $E\in \pP_{t_0}(\phi)$, 
the HN filtrations
with respect to $\gamma(t_{\pm})$
yield 
short exact sequences in $\aA_{k}$
respectively,  
\begin{align}\label{ex1}
0\to E_1 \to E \to E_2 \to 0, \\
\label{ex2}
0 \to E_1' \to E \to E_2' \to 0,
\end{align}
satisfying the following. 
\begin{itemize}
\item $E_1 \in \langle \oO_X[-k] \rangle_{\ex}$ and 
$E_2' \in \langle \oO_X[-k] \rangle_{\ex}$. 
\item $E_2 \in \pP_{t_{+}}(\phi)$
 and $E_{1}' \in \pP_{t_{-}}(\phi)$. 
\end{itemize}
Conversely if an object $E\in \aA_{k}$ 
fits into an exact sequence (\ref{ex1}) or (\ref{ex2})
satisfying the above conditions, then $E\in \pP_{t_0}(\phi)$. 
\end{prop}
\begin{proof}
For simplicity we only see the sequence (\ref{ex1}). 
The result follows from Lemma~\ref{below}
if $\cl(E) \in \Gamma_0$, therefore we assume  
 that $\cl(E) \in \Gamma \setminus \Gamma_{0}$. 
Also we may assume that $E$ is not $Z_{t_{+}}$-semistable,
hence there is an exact sequence in $\aA_{k}$,
\begin{align*}
0 \to E' \to E \to E'' \to 0,
\end{align*}
such that $\arg Z_{t_{+}}(E')> \arg Z_{t_{+}}(E'')$.
Then the same argument in the proof of 
Lemma~\ref{fixed} shows that $\cl(E')$ or $\cl(E'')$
is contained in
$\langle \oO_X[-k] \rangle_{\ex}$. 
This implies that 
the HN filtration w.r.t. $\gamma(t_{+})$ consists 
of a two step filtration in $\aA_k$, which 
we denote by (\ref{ex1}), 
and either $E_1$ or $E_2$ is contained in 
$\langle \oO_X[-k] \rangle_{\ex}$. 
 Since $t_{+}>t_{0}$, 
we have 
\begin{align*}
\arg Z_{t_{+}}(\oO_X[-k])
&=\pi(t_{+}-k) \\
&>\pi(t_0-k) \\
&= \arg Z_{t_0}(E) \\
&=\arg Z_{t_{+}}(E).
\end{align*}
Therefore 
we must have $\cl(E_1) \in \langle \oO_X[-k] \rangle_{\ex}$
and $E_2 \in \pP_{t_{+}}(\phi)$.  

Conversely suppose that $E\in \aA_{k}$ fits into an exact sequence (\ref{ex1}). 
By Lemma~\ref{below}, Lemma~\ref{fixed}
and noting Remark~\ref{rmk216}, we have 
\begin{align*}
\pP_{t_{+}}(\phi) \subset \pP_{t_{0}}(\phi), \quad 
\oO_X[-k] \in \pP_{t_{0}}(\phi). 
\end{align*}
Therefore we have $E\in \pP_{t_{0}}(\phi)$. 
\end{proof}

\begin{rmk}
As we discussed in the introduction, 
the image of $\iota$ in the diagram (\ref{iecom}) 
may be interpreted as a loop 
around the conifold point in Figure~\ref{fig:one}
if $X$ is a quintic 3-fold, 
since the covering transformation of the left arrow of (\ref{iecom})
is induced by the action of $\Phi_{\oO_X}$. 
\end{rmk}

\begin{rmk}
Although the
result of Theorem~\ref{obtain} holds 
under the assumption (\ref{add}), 
the continuous map $\gamma$ exists 
without that assumption
once we fix an ample divisor $\omega$. 
The results of Lemma~\ref{fixed} and Proposition~\ref{prop1} 
hold as well without (\ref{add}). 
\end{rmk}

\section{Donaldson-Thomas theory}
In this section, we introduce generalized 
DT invariants counting semistable objects in $\dD_X$
with respect to our weak stability conditions, 
and establish their wall-crossing formula.  
Originally DT theory is introduced in~\cite{Thom} 
as counting stable coherent sheaves
on Calabi-Yau 3-folds, 
and defined only when semistable 
sheaves and stable sheaves coincide. 
The generalized DT theory 
introduced by Joyce-Song~\cite{JS}
is also defined when there is a
semistable but not stable sheaf, and 
the notion of Hall algebra
is used for the definition.   
The same construction 
is also applied in our situation, 
and we first give some 
notions needed for the definition. 
\subsection{Hall algebra}
In this subsection, we recall 
the notion of Hall algebra
via moduli stacks. 
See~\cite{GL} for the introduction of 
stacks and~\cite{Joy2} for the detail on the 
Hall algebra. 

Let $\mM$ be the moduli stack of 
objects $E\in D^b(\Coh(X))$ satisfying 
\begin{align}\label{neg:ex}
\Ext^i_X(E, E)=0, \quad i<0. 
\end{align}
By the result of Lieblich~\cite{LIE}, 
$\mM$ is an algebraic stack 
locally of finite type over $\mathbb{C}$. 
For each element
\begin{align*}
\sigma=(Z, \cC) \in \Stab_{\Gamma_{\bullet}}^{\circ}(\dD_X),
\end{align*}
where $\cC \subset \dD_X$
is the heart of a bounded t-structure, 
we have the (abstract) substack, 
\begin{align}\notag
\oO bj(\cC) \subset \mM,
\end{align}
which parameterizes objects $E\in \cC$.
The stack $\oO bj(\cC)$
decomposes as follows, 
\begin{align*}
\oO bj(\cC)=\coprod_{v\in \Gamma}
\oO bj^v(\cC),
\end{align*}
where $\oO bj^{v}(\cC)$
is the stack of objects $E\in \cC$
with $\cl(E)=v$. 

Suppose for instance that $\oO bj(\cC)$
is an algebraic stack 
locally of finite type over
$\mathbb{C}$. 
The $\mathbb{Q}$-vector space
$\hH(\cC)$ is generated by the 
isomorphism classes of symbols, 
\begin{align*}
[\xX \stackrel{f}{\to} \oO bj(\cC)],
\end{align*}
where $\xX$ is an algebraic stack 
of finite type over $\mathbb{C}$, and 
$f$ is a 1-morphism of stacks. 
Here two symbols 
$[\xX_i \stackrel{f_i}{\to} \oO bj(\cC)]$
for $i=1, 2$
are \textit{isomorphic} if there is 
an isomorphism of stacks, 
\begin{align*}
g\colon \xX_1 \stackrel{\sim}{\to} \xX_2,
\end{align*}
such that $f_2 \circ g \cong f_1$. 
The relations are generated by 
\begin{align*}
[\xX \stackrel{f}{\to} \oO bj(\cC)]
\sim [\uU \stackrel{f|_{\uU}}{\to} \oO bj(\cC)]
+ [\zZ \stackrel{f|_{\zZ}}{\to} \oO bj(\cC)],
\end{align*}
where $\uU \subset \xX$ is an open substack 
and $\zZ =\xX \setminus \uU$. 

Let $\eE x(\cC)$ be the stack of 
short exact sequence in $\cC$, 
\begin{align}\label{short}
0 \to A_1 \to A_2 \to A_3 \to 0. 
\end{align}
By sending the exact sequence (\ref{short})
to the object $A_i$, we obtain morphisms, 
\begin{align*}
p_i \colon \eE x(\cC) \to \oO bj(\cC), \quad 
i=1, 2, 3. 
\end{align*}
For two elements 
\begin{align*}
\rho_i=[\xX_i \stackrel{f_i}{\to} \oO bj(\cC)] \in \hH(\cC), 
\quad i=1, 2, 
\end{align*}
we have the diagram, 
\begin{align*}
\xymatrix{
\zZ \ar[r]^{h} \ar[d] & \eE x(\cC) \ar[r]^{p_2} \ar[d]^{(p_1, p_3)} & 
\oO bj(\cC) \\
\xX_1 \times \xX_2 \ar[r]^{(f_1, f_2)} & \oO bj(\cC)^{\times 2}.
& }
\end{align*}
Here the left diagram is a Cartesian square. 
We define the $\ast$-product
$\rho_1 \ast \rho_2$
to be
\begin{align*}
\rho_1 \ast \rho_2 \cneq [\zZ \stackrel{p_2 \circ h}{\to} \oO bj(\cC)]
\in \hH(\cC). 
\end{align*}
It is proved in~\cite[Theorem~5.2]{Joy2} that
$\ast$ is an associative product on $\hH(\cC)$
with unit given by 
\begin{align*}
1=[\Spec \mathbb{C} \to \oO bj(\cC)],
\end{align*}
whose image corresponds to $0\in \cC$. 
The algebra $\hH(\cC)$ is $\Gamma$-graded,
\begin{align*}
\hH(\cC)=\bigoplus_{v\in \Gamma} \hH^v(\cC),
\end{align*}
where $\hH^v(\cC)$
is spanned by symbols 
$[\xX \stackrel{f}{\to}\oO bj(\cC)]$
such that $f$ factors through 
the substack $\oO bj^{v}(\cC) \subset \oO bj(\cC)$. 

We will use certain completions of the 
algebra $\hH(\cC)$. 
Let $V\subset \mathbb{H}$ be a subset
written as 
\begin{align}\label{defV}
V=\{ r\exp(i\pi \phi) : r>0, \ \phi_1 \le \phi \le \phi_2\}. 
\end{align}
for some $\phi_1, \phi_2 \in \mathbb{R}$
with $0\le \phi_2 -\phi_1 <1$. 
We define $\widehat{\hH}(\cC)_{Z, V}$
to be
\begin{align*}
\widehat{\hH}(\cC)_{Z, V}
\cneq \prod_{v\in \Gamma, \ Z(v) \in V}\hH^v(\cC). 
\end{align*}

\subsection{Elements $\delta^v{(Z)}$ and $\epsilon^v(Z)$}
For $\sigma=(Z, \cC) \in \Stab_{\Gamma_{\bullet}}^{\circ}(\dD_X)$
and $v\in \Gamma$, 
the stack of $Z$-semistable 
objects $E\in \cC$ with $\cl(E)=v$ is denoted by, 
\begin{align*}
\mM^v(Z) \subset \oO bj^{v}(\cC).
\end{align*}
Suppose for instance that $\mM^v(Z)$ is an 
algebraic stack of finite type 
over $\mathbb{C}$. 
Then the above stack defines the element, 
\begin{align*}
\delta^{v}(Z)\cneq [\mM^v(Z) \hookrightarrow \oO bj^v(\cC)]
\in \hH^v(\cC).
\end{align*}
We say a subset $l\subset \mathbb{H}$
as a \textit{ray} if there is $\phi \in (0, 1]$
such that $l=\mathbb{R}_{>0}\exp(i\pi \phi)$. 
For each ray $l$, we define $\delta^{l}(Z)$ 
to be 
\begin{align}\label{dray}
\delta^{l}(Z)\cneq 1+\sum_{Z(v)\in l} \delta^{v}(Z)
\in \widehat{\hH}(\cC)_{Z, l}. 
\end{align}
Then we define the element $\epsilon^{l}(Z)$ to be 
\begin{align*}
\epsilon^{l}(Z)=\log \delta^{l}(Z) \in \widehat{\hH}(\cC)_{Z, l}.
\end{align*}
Namely $\epsilon^{l}(Z)$ is given by 
\begin{align*}
\epsilon^{l}(Z) \cneq \sum_{Z(v) \in l}\epsilon^{v}(Z),
\end{align*}
where $\epsilon^v(Z)$ is written as 
\begin{align}\label{sum}
\epsilon^{v}(Z)=\sum_{\begin{subarray}{c}
m \ge 1, \
v_1, \cdots, v_m \in \Gamma, \\
Z(v_i) \in \mathbb{R}_{>0}Z(v), \\
v_1+\cdots +v_m=v. 
\end{subarray}}\frac{(-1)^{m-1}}{l}
\delta^{v_1}(Z) \ast \cdots \ast \delta^{v_m}(Z).
\end{align}
The above definition makes sense by the following 
lemma, whose proof will be given in Section~\ref{sec:tech}.  
\begin{lem}\label{finsum}
The sum (\ref{sum}) is a finite sum. 
\end{lem}
So far we have assumed that the 
stacks $\oO bj(\cC)$ and $\mM^v(Z)$ 
are algebraic stacks locally of finite type, 
finite type respectively. 
However these are too strong conditions 
for the applications. 
In fact it is enough to show the following lemma, 
by discussing with 
the framework of 
Kontsevich-Soibelman~\cite[Section~3]{K-S}.  
The proof will be given in Section~\ref{sec:tech}. 
\begin{lem}\label{const}
For any $\sigma=(Z, \cC) \in \Stab_{\Gamma_{\bullet}}^{\circ}(\dD_X)$
and $v\in \Gamma$, we have the following. 

(i) The $\mathbb{C}$-valued points 
of the substack $\oO bj^{v}(\cC) \subset \mM$ are 
countable union of constructible subsets in $\mM$. 

(ii) The $\mathbb{C}$-valued points 
of the substack $\mM^v(Z) \subset \mM$
form a constructible subset in $\mM$. 
\end{lem}

\subsection{Generalized DT invariants}
For a quasi-projective variety $Y$, 
we define
\begin{align*}
\Upsilon(Y) \cneq \sum_{j, k \ge 0} (-1)^k
\dim W_j( H_{c}^{k}(Y, \mathbb{Q}))q^{j/2} \in \mathbb{Q}[q^{1/2}],
\end{align*}
where $W_{\bullet}$ is the weight 
filtration on the compact support cohomology
group $H_{c}^{\ast}(Y, \mathbb{Q})$. 
The assignment $Y \mapsto \Upsilon(Y)$ extends to
the Hall algebra 
$\hH(\cC)$, 
\begin{align*}
\Upsilon \colon \hH(\cC) \to \mathbb{Q}(q^{1/2}), 
\end{align*}
such that we have 
\begin{align*}
\Upsilon ([[Y/G] \to \oO bj(\cC)])=\Upsilon (Y)/ \Upsilon(G), 
\end{align*}
where $Y$ is a quasi-projective variety,
$G$ is a special algebraic group 
acting on $G$ and $[Y/G]$ is the quotient 
stack with respect to the $G$-action.  
(cf.~\cite[Theorem~4.9]{Joy5}.)
Here an algebraic group $G$ is \textit{special} if 
any principle $G$-bundle is Zariski locally trivial.
(cf.~\cite[Definition~2.1]{Joy5}.) 

Recall that for any variety $Y$, there is a canonical 
constructible function by Behrend~\cite{Beh},
\begin{align*}
\nu_{Y} \colon Y \to \mathbb{Z}. 
\end{align*}
The above function satisfies the following 
properties. 
\begin{itemize}
\item 
For $p\in Y$, 
suppose that there is an 
analytic open neighborhood 
$p\in U \subset Y$, 
a holomorphic function 
$f\colon V \to \mathbb{C}$ 
on a complex manifold $V$ 
such that $U\cong \{df=0\}$. 
Then we have 
\begin{align}\label{Miln}
\nu_{Y}(p)=(-1)^{\dim V}(1-\chi(M_f(p))).
\end{align}
Here $M_f(p)$ is a Milnor fiber of $f$ at $p$. 
\item If there is a symmetric perfect obstruction 
theory on $Y$, we have 
\begin{align*}
\int_{[Y]^{\rm{vir}}}1 =
\sum_{m\in \mathbb{Z}}m \cdot \chi(\nu_{Y}^{-1}(m)).
\end{align*}
\end{itemize}
The Behrend's constructible function can be 
generalized to any algebraic stack $\yY$,  
\begin{align*}
\nu_{\yY} \colon \yY \to \mathbb{Z}.
\end{align*}
Namely if $\yY$ is written as a global 
quotient stack $\yY=[Y/G]$, then 
$\nu_{\yY}=(-1)^{\dim G}\nu_{Y}$. 
For a general case, the existence of $\nu_{\yY}$
is proved in~\cite[Proposition~4.4]{JS}. 

Let $\mM$ be the moduli stack of objects $E\in D^b(\Coh(X))$
satisfying (\ref{neg:ex}). 
By the above argument, there is Behrend's
constructible
function $\nu_{\mM}$ on $\mM$. 
The function $\nu_{\mM}$ should be calculated 
by the Euler characteristic of some 
holomorphic function as in (\ref{Miln}). 
In fact the following conjecture, 
which is a derived category version of~\cite[Theorem~5.5]{JS}, 
should be true. 
\begin{conj}\label{conjBG}
For any $[E] \in \mM(\mathbb{C})$, 
let $G$ be a maximal reductive 
subgroup in $\Aut(E)$.
Then there exists a $G$-invariant analytic 
open neighborhood $V$ of $0$ in 
$\Ext^1(E, E)$, 
a $G$-invariant holomorphic function $f\colon V\to \mathbb{C}$
with $f(0)=df|_{0}=0$, and a smooth morphism 
of complex analytic stacks
$\Phi \colon [\{df=0\}/G] \to \mM$
of relative dimension $\dim \Aut(E)- \dim G$. 
\end{conj}
The above conjecture is
proved in~\cite[Theorem~5.5]{JS} if $E\in \Coh(X)$. 
We believe that similar arguments show
the above conjecture for any $[E] \in \mM(\mathbb{C})$, although
several details have to be checked. 
Also Behrend-Getzler~\cite{BG}
have announced a similar result, 
so in what follows we assume that the above 
conjecture is true.

The Behrend 
function on $\mM$ defines the map
$\nu \cdot  \colon \hH(\cC) \to \hH(\cC)$, 
\begin{align*}
\nu \cdot ([\xX \stackrel{f}{\to} \oO bj(\cC)]) 
\cneq \sum_{m \in \mathbb{Z}}m \cdot [\xX \times _{\oO bj(\cC)}
\nu_{\mM}|_{\oO bj(\cC)}^{-1}(m) \to \oO bj(\cC)]. 
\end{align*}
The generalized DT invariant is defined as follows. 
\begin{defi}\emph{\bf{\cite[Definition~5.13]{JS}}}\label{def:DT}
\emph{We define $\DT_{Z}(v)$ as follows, 
\begin{align*}
\DT_{Z}(v)\cneq -\lim_{q^{1/2} \to 1} (q-1)
\Upsilon(\nu \cdot \epsilon^v(Z)) \in \mathbb{Q}. 
\end{align*}}
\end{defi}
The existence of the limit 
is essentially proved in~\cite[Section~6.2]{Joy5}. 
\subsection{Lie algebra homomorphism}
Let $\chi \colon 
\Gamma \times \Gamma \to \mathbb{Z}$ be 
an anti-symmetric bilinear form on $\Gamma$, 
given by 
\begin{align*}
\chi((r, \beta, n), (r', \beta', n'))=rn'-r'n. 
\end{align*}
By the Riemann-Roch theorem and the Serre duality, we have 
\begin{align}\notag
\chi(\cl(E), \cl(F))=&
\dim \Hom(E, F)-\dim \Ext^1(E, F) \\
\label{RRS}
&+\dim \Ext^1(F, E)-\dim \Hom(F, E),
\end{align}
for $E, F \in \dD_X$. 

Let $\mathfrak{g}$ be the $\mathbb{Q}$-vector space 
spanned by symbols $c_v$ for $v\in \Gamma$, 
\begin{align*}
\mathfrak{g}=\bigoplus_{v\in \Gamma} \mathbb{Q}c_v. 
\end{align*}
There is a Lie-algebra structure on $\mathfrak{g}$
with bracket given by 
\begin{align*}
[c_v, c_{v'}]=(-1)^{\chi(v, v')}\chi(v, v') c_{v+v'}. 
\end{align*}

For a weak stability condition 
$\sigma=(Z, \cC) \in \Stab_{\Gamma_{\bullet}}^{\circ}(\dD_X)$,
we can define the Lie algebra of 
\textit{virtual indecomposable objects}, 
\begin{align*}
\mathfrak{H}(\cC) \subset \hH(\cC), 
\end{align*}
in the same way of~\cite[Definition~5.14]{Joy2}. 
The definitions of
virtual indecomposable objects and 
the Lie algebra $\mathfrak{H}(\cC)$
are complicated, and we omit the detail.
The Lie algebra $\mathfrak{H}(\cC)$ is also 
$\Gamma$-graded, 
\begin{align*}
\mathfrak{H}(\cC)=\bigoplus_{v\in \Gamma} \mathfrak{H}^v(\cC), 
\quad \mathfrak{H}^v(\cC)=\hH^{v}(\cC) \cap \mathfrak{H}(\cC).
\end{align*}
For $v\in \Gamma$, the 
element $\delta^v(Z)$ is not necessary 
virtual indecomposable, but 
$\epsilon^v(Z)$ is always virtual indecomposable. 
(cf.~\cite[Theorem~8.7]{Joy3}.) 
Assuming Conjecture~\ref{conjBG}, 
the following result can 
be proved along the same 
way of~\cite[Theorem~5.14]{JS}. 
\begin{thm}\emph{{\bf \cite[Theorem~5.14]{JS}}} \label{thm:Lie}
There is a homomorphism as
 $\Gamma$-graded Lie algebras, 
\begin{align}\label{Psi}
\Psi \colon \mathfrak{H}(\cC) \to \mathfrak{g}, 
\end{align}
which takes $\epsilon^v(Z)$ to $-\DT_{Z}(v) c_v$. 
\end{thm}
Let $V \subset \mathbb{H}$ be a subset 
defined by (\ref{defV}). 
We can similarly define the completions 
of the Lie algebras, 
\begin{align*}
\widehat{\mathfrak{H}}(\cC)_{Z, V}
&\cneq \prod_{v\in \Gamma, \ Z(v) \in V}
\mathfrak{H}^v(\cC), \\
\widehat{\mathfrak{g}}(\cC)_{Z, V}
&\cneq \prod_{v\in \Gamma, \ Z(v) \in V}
\mathfrak{g}^v(\cC),
\end{align*}
and (\ref{Psi}) 
induces the Lie algebra homomorphism, 
\begin{align}\label{Psi2}
\Psi \colon \widehat{\mathfrak{H}}(\cC)_{Z, V}
\to \widehat{\mathfrak{g}}(\cC)_{Z, V}.
\end{align}

\subsection{DT invariants around the conifold point}
For $t\in (k, k+1]$ with $k\in \mathbb{Z}$, let 
\begin{align*}
\gamma(t)=(Z_t, \aA_k)=(Z_t, \pP_t) \in \Stab_{\Gamma_{\bullet}}(\dD_X),
\end{align*}
be the weak stability condition
 defined in (\ref{map:g}).
Here $\aA_k$ is the heart of a bounded t-structure 
given by (\ref{def:Ak}) and $\pP_t$ is the
associated slicing.  
For an element $v \in \Gamma$, 
the associated element, 
$\epsilon^{v}(Z_t) \in \hH(\aA_k)$
defines the invariant, 
\begin{align*}
\DT_{Z_t}(v)
=-\lim_{q^{1/2}\to 1}(q-1)\Upsilon(\nu \cdot \epsilon^{v}(Z_t)),
\end{align*}
as in the same way of Definition~\ref{def:DT}. 
\begin{defi}\label{def:around}
\emph{For data 
\begin{align*}
(r, \beta, n) \in \Gamma, \quad 
t\in \mathbb{R}, \quad \phi \in \mathbb{R},
\end{align*}
we define the invariant
$\DT_t(r, \beta, n, \phi) \in \mathbb{Q}$ as follows.}

\emph{When $0<\phi \le 1$, suppose that the following holds,
\begin{align}\label{suppose}
Z_t(r, -\beta, -n) \in \mathbb{R}_{>0}\exp(i\pi \phi). 
\end{align}
Then we define 
\begin{align*}
\DT_t(r, \beta, n, \phi) 
\cneq \DT_{Z_t}(r, -\beta, -n). 
\end{align*}
If (\ref{suppose}) is not satisfied, 
we set $\DT_t(r, \beta, n, \phi)=0$.} 

\emph{In a general case,
 writing $\phi=m+\phi_0$
with $m\in \mathbb{Z}$ and $0<\phi_0 \le 1$, we define} 
\begin{align}\label{general}
\DT_t(r, \beta, n, \phi)\cneq
\DT_t((-1)^m r, (-1)^m \beta, (-1)^m n, \phi_0). 
\end{align}
\end{defi}
Note that $\DT_t(r, \beta, n, \phi)$
is a counting invariant of objects
$E\in \dD_X$ satisfying 
\begin{align*}
E\in \pP_t(\phi), \quad 
\cl(E)=(r, -\beta, -n). 
\end{align*}
In case of $\beta=n=0$, 
the invariant is already computed. 
\begin{lem}\label{already}
For 
$0<\phi \le 1$ and 
$t\in (k, k+1]$ with $k\in \mathbb{Z}$, 
we have 
\begin{align*}
\DT_t(r, 0, 0, \phi)=\left\{
\begin{array}{ll}
\frac{1}{r^2}, & \emph{\mbox{ if }}
t=\phi+k,  (-1)^k r>0, \\
0, & \emph{\mbox{ otherwise. }}
\end{array}
\right.
\end{align*}
\end{lem}
\begin{proof}
As in the previous subsection, 
let $\mM^{(r, 0, 0)}(Z_t)$ be the 
substack of $\oO bj(\aA_k)$, 
which parameterizes $Z_t$-semistable objects
$E\in \aA_k$
with $\cl(E)=(r, 0, 0)$. 
By Lemma~\ref{below}
and the assumption (\ref{assum:van}),  
we have the isomorphism of stacks, 
\begin{align*}
\mM^{(r, 0, 0)}(Z_t) \cong [\Spec \mathbb{C}/\GL_{r'}(\mathbb{C})],
\end{align*}
where $r'=(-1)^k r$. 
The unique $\mathbb{C}$-valued point 
of the RHS corresponds to the object
 $\oO_X[-k]^{\oplus r'} \in \aA_k$, 
and it has phase $t-k$
by Lemma~\ref{below}. 
Hence $\DT_t(r, 0, 0, \phi)$ is non-zero 
only if $t=\phi+k$, and the 
 contribution of the object
$\oO_X[-k]^{\oplus r'}$ 
is computed in the same way of~\cite[Example~6.2]{JS}. 
\end{proof}
We set the following generating series, 
\begin{align*}
\DT_{t}(\phi)\cneq \sum_{(r, \beta, n)\in \Gamma}
\DT_{t}(r, \beta, n, \phi)x^r y^{\beta} z^n. 
\end{align*}
We have the following lemma. 
\begin{lem}\label{sense}
For a fixed $\phi \in \mathbb{R}$ and 
$k\in \mathbb{Z}$, 
the generating series $\DT_t(\phi)$
does not depend on $t\in (\phi+k, \phi+k+1)$. 
\end{lem}
\begin{proof}
The result immediately follows from Lemma~\ref{fixed}. 
\end{proof}
In what follows, we set 
\begin{align*}
\DT^{k}(r, \beta, n, \phi) &\cneq 
\DT_t(r, \beta, n, \phi), \\
\DT^{k}(\phi) &\cneq \DT_{t}(\phi), 
\end{align*}
if $t\in (\phi+k, \phi+k+1)$ with $k\in \mathbb{Z}$. 
The above notation makes sense by Lemma~\ref{sense}. 

\subsection{Wall-crossing formula}
In this subsection, 
we give a proof of 
the following theorem, 
assuming Conjecture~\ref{conjBG}. 
\begin{thm}\label{thm:wall}
For $\phi \in \mathbb{R}$
and $k\in \mathbb{Z}$, the series 
$\DT^{k}(\phi)$ is obtained from $\DT^{k-1}(\phi)$
by the following transformation, 
\emph{\begin{align*}
z^n \mapsto \left\{ \begin{array}{cc}
(1-(-1)^{n}x)^{n}z^n, & \mbox{ if }k \mbox{ is even.} \\
x^n z^n/(1-(-1)^n x)^{n}, & \mbox{ if }k\mbox{ is odd.}
\end{array} \right. 
\end{align*}}
\end{thm}
\begin{proof}
Since $\DT_t(\phi)=\DT_{t}(\phi+2)$, we may 
assume that $0<\phi \le 2$. 
First we discuss
the case of $0<\phi<1$. 
Let us take $k\in \mathbb{Z}$
and set 
 $t_0=\phi+k$. 
We take $0<\varepsilon \ll 1$
so that $(t_0-\varepsilon, t_0+\varepsilon) \subset (k, k+1)$
holds. 
We set $t_{\pm}=t_{0}\pm \varepsilon$, 
and $V \subset \mathbb{H}$ to be  
\begin{align*}
V=\{ r\exp(i\pi \theta) : r>0, \ \theta \in
 [\phi-\varepsilon, \phi+\varepsilon]\}. 
\end{align*}
For each $t\in (k, k+1]$, 
the proof of Lemma~\ref{already} shows 
that  
\begin{align*}
\delta^{((-1)^k r, 0, 0)}(Z_{t})
=[[\Spec \mathbb{C}/ \GL_{r}(\mathbb{C})] \to \oO bj(\aA_k)]
\in \hH(\aA_k),
\end{align*}
which corresponds to the object $\oO_X[-k]^{\oplus r}$. 
The above element of $\hH(\aA_k)$
does not depend on $t\in (k, k+1]$, 
and we denote it by $\delta^{((-1)^k r, 0, 0)}$
for simplicity. 
We define the element $\delta^k$
and $\epsilon^k$ to be 
\begin{align*}
\delta^k &\cneq 
1+\sum_{r\ge 1}\delta^{((-1)^k r, 0, 0)} \in 
\widehat{\hH}(\aA_k)_{Z_{t_0}, V}, \\
\epsilon^k &\cneq \log \delta^k
 \in 
\widehat{\hH}(\aA_k)_{Z_{t_0}, V}.
\end{align*}
The element $\epsilon^k$ is shown to be well-defined
by the same way of Lemma~\ref{finsum}.  
Let us set the ray $l=\mathbb{R}_{>0} \exp(i\pi \phi)$. 
By the same argument of~\cite[Theorem~5.11]{Joy4}, 
the result of Proposition~\ref{prop1}
can be expressed in terms of
a relationship in the completed Hall algebra
$\widehat{\hH}(\aA_k)_{Z_{t_0}, V}$,
\begin{align}\label{deltak}
\delta^k \ast \delta^{l}(Z_{t_{+}})
=\delta^{l}(Z_{t_0})
=\delta^{l}(Z_{t_{-}}) \ast \delta^k.  
\end{align}
Therefore we obtain the formula
in $\widehat{\hH}(\aA_k)_{Z_{t_0}, V}$,
\begin{align*}
\exp(\epsilon^l(Z_{t_{+}}))=
\exp(\epsilon^k)^{-1}\ast 
\exp(\epsilon^l(Z_{t_{-}}))
\ast \exp(\epsilon^k). 
\end{align*}
Here $\delta^l(Z_t)$ is defined 
in (\ref{dray}). 
Now we can apply a 
 version of the Baker-Campbell-Hausdorff 
formula, and 
the RHS coincides with 
\begin{align*}
\exp \left(
\epsilon^{l}(Z_{t_{-}})
+\sum_{m \ge 1}\frac{(-1)^m}{m!}
\mathrm{Ad}_{\epsilon^k}^{m}(\epsilon^{l}(Z_{t_{-}})) \right). 
\end{align*}
Here we have set 
\begin{align*}
\mathrm{Ad}_{\epsilon^k}^{m}(\epsilon^{l}(Z_{t_{-}})) 
=\overbrace{[\epsilon^k [\epsilon^k [
\cdots [\epsilon^{k}}^{m}, \epsilon^{l}(Z_{t-})] \cdots ]]. 
\end{align*}
By taking the logarithms of both sides, 
we obtain the equality
in $\widehat{\hH}(\aA_k)_{Z_{t_0}, V}$,
\begin{align*}
\epsilon^{l}(Z_{t_{+}})=
\epsilon^{l}(Z_{t_{-}})
+\sum_{m \ge 1}\frac{(-1)^m}{m!}
\mathrm{Ad}_{\epsilon^k}^{m}(\epsilon^{l}(Z_{t_{-}})). 
\end{align*}
Let us set 
\begin{align*}
\mathfrak{DT}_{t}(\phi)
&=\sum_{\begin{subarray}{c}
(r, -\beta, -n) \in \Gamma, \\
Z_{t}(r, -\beta, -n) \in l
\end{subarray}}
\DT_{t}(r, \beta, n, \phi) c_{(r, -\beta, -n)}
\in \widehat{\mathfrak{g}}_{Z_{t_0}, V}, \\
\mathfrak{E}^{k} &=\sum_{r\ge 1}\frac{1}{r^2}c_{((-1)^k r, 0, 0)}
\in \widehat{\mathfrak{g}}_{Z_{t_0}, V}.
\end{align*}
Applying the 
Lie algebra homomorphism (\ref{Psi2})
and using Lemma~\ref{already},
we obtain the equality
in $\widehat{\mathfrak{g}}_{Z_{t_0}, V}$, 
\begin{align*}
\mathfrak{DT}_{t_{+}}(\phi)
=\mathfrak{DT}_{t_{-}}(\phi)
+\sum_{m \ge 1}\frac{1}{m!}
\mathrm{Ad}_{\mathfrak{E}^k}^{m}(\mathfrak{DT}_{t_{-}}(\phi)). 
\end{align*} 
By expanding the RHS, we can easily obtain the following, 
\begin{align}\notag
\DT_{t_{+}}(r, \beta, n, \phi)
&=\DT_{t_{-}}(r, \beta, n, \phi) + \\
\label{expand}
&\sum_{\begin{subarray}{c}
m\ge 1, \\
r_0 \in \mathbb{Z}, 
r_1, \cdots, r_m \ge 1, \\
r_{0}+
\sum_{i=1}^{m}(-1)^k r_i =r.
\end{subarray}}
\frac{(-1)^{n(\sum_{i=1}^{m}r_i)+m(k+1)}n^m}{m! \prod_{i=1}^{m}r_i}
\DT_{t_{-}}(r_{0}, \beta, n, \phi). 
\end{align}
For a fixed $(\beta, n) \in H_2 \oplus H_0$, 
we set 
\begin{align*}
\DT_{t}(\beta, n, \phi)=
\sum_{r\in \mathbb{Z}}
\DT_{t}(r, \beta, n, \phi)x^r. 
\end{align*}
Then the equality (\ref{expand}) implies that 
\begin{align*}
&\DT_{t_{+}}(\beta, n, \phi)
= \\
&\qquad \left(
\sum_{m\ge 0}\frac{1}{m!}
\sum_{r_1, \cdots, r_m \ge 1}
\prod_{i=1}^{m}
\frac{(-1)^{k+1}n}{r_i}
\{(-1)^n x \}^{(-1)^k r_i}
  \right) \DT_{t_{-}}(\beta, n, \phi). 
\end{align*}
Then the assertion follows from Lemma~\ref{ass}
below.  

When $\phi=1$, we consider the rotated 
weak stability condition, 
\begin{align*}
\frac{1}{2} \cdot \gamma(t_0)=
(-i Z_{t_0}, \pP_{t_0}((1/2, 3/2])). 
\end{align*}
Then the equality similar to (\ref{deltak})
holds in the Hall algebra of $\pP_{t_0}((1/2, 3/2])$. 
The argument for $0<\phi<1$ is also 
applied in this situation, 
and we obtain the same wall-crossing formula. 

Finally when $1<\phi \le 2$, then the assertion holds
noting that 
\begin{align*}
\DT^{k}(r, \beta, n, \phi)=\DT^{k+1}(-r, -\beta, -n, \phi-1). 
\end{align*}
\end{proof}
We have used the following lemma. 
\begin{lem}\label{ass}
\emph{\begin{align}\notag 
\sum_{m\ge 0}\frac{1}{m!}
\sum_{r_1, \cdots, r_m \ge 1} &
\prod_{i=1}^{m}
\frac{(-1)^{k+1}n}{r_i}
\{(-1)^n x \}^{(-1)^k r_i}
\\ \label{simplify}
&=\left\{ \begin{array}{cc}
(1-(-1)^n x)^n, & \mbox{ if }k \mbox{ is even, } \\
x^n/(1-(-1)^n x)^n, & \mbox{ if }k \mbox{ is odd. }
\end{array}
\right. 
\end{align}}
\end{lem}
\begin{proof}
We can calculate as follows. 
\begin{align*}
\sum_{m\ge 0}\frac{1}{m!}
\sum_{r_1, \cdots, r_m \ge 1} &
\prod_{i=1}^{m}
\frac{(-1)^{k+1}n}{r_i}
\{(-1)^n x \}^{(-1)^k r_i}
\\
&=
\exp \left(
\sum_{r\ge 1}\frac{(-1)^{k+1}n}{r} \{(-1)^n x\}^{(-1)^k r}
 \right) \\
&= 
\exp \log \left( 1-\{ (-1)^n x \}^{(-1)^k}  \right)^{(-1)^{k}n} \\
&= \left( 1-\{(-1)^n x \}^{(-1)^k} \right)^{(-1)^k n}.
\end{align*}
The last one is written as the RHS of (\ref{simplify}). 
\end{proof}

\begin{rmk}\label{modul}
By Theorem~\ref{thm:wall}, 
the series $\DT^{k+2}(\phi)$
is obtained from $\DT^k(\phi)$ 
by the variable change $z \mapsto xz$. 
On the other hand,  $\DT^{k+2}(\phi)$
and $\DT^k(\phi)$ are
related by the equivalence $\Phi_{\oO_X}$
by our construction of $\gamma$,
and the variable change by $\Phi_{\oO_X}$
is given by $z \mapsto xz$
by Lemma~\ref{com2}.
This indicates that the wall-crossing 
formula does not indicate any 
modularity of our invariants. 
This is unfortunate in some sense, 
since there are situations in which 
the wall-crossing formula indicates 
some modularity of the invariants, 
e.g. the invariants on K3 surfaces~\cite{Tst3}. 
\end{rmk}

\subsection{Euler characteristic version}
We can also investigate
the Euler characteristic 
version of our invariants, 
which are defined in a similar
way to $\DT_t(r, \beta, n, \phi)$
without the Behrend function. 
For $(r, \beta, n) \in \Gamma$
and $t\in \mathbb{R}$, suppose that 
(\ref{suppose}) holds. 
When $0<\phi \le 1$, we define  
\begin{align*}
\Eu_{t}(r, \beta, n, \phi) \cneq 
\lim_{q^{1/2} \to 1}(q-1) 
\Upsilon(\epsilon^{(r, -\beta, -n)}(Z_t))
\in \mathbb{Q},
\end{align*}
and set $\Eu_t(r, \beta, n, \phi)=0$ if (\ref{suppose})
is not satisfied. 
Here recall that we defined  
 $\epsilon^{(r, -\beta, -n)}(Z_t)$ 
as an element of $\hH(\aA_k)$
if $t\in (k, k+1]$ for $k\in \mathbb{Z}$. 
For a general $\phi \in \mathbb{R}$, 
writing $\phi=m+\phi_0$
with $0<\phi_0 \le 1$, we define 
\begin{align*}
\Eu_{t}(r, \beta, n, \phi)
\cneq \Eu_{t}((-1)^m r, (-1)^m \beta, (-1)^m n, \phi_0). 
\end{align*}
The generating series is also defined as well, 
\begin{align*}
\Eu_{t}(\phi) \cneq 
\sum_{(r, \beta, n)\in \Gamma}
\Eu_t(r, \beta, n, \phi)x^r y^{\beta}z^n. 
\end{align*}
The following theorem can be proved 
in a similarly way of Theorem~\ref{thm:wall}, 
using a version of~\cite[Theorem~6.12]{Joy2}
instead of Theorem~\ref{thm:Lie}.
\begin{thm}\label{thm:Eu}
(i) For a given $k\in \mathbb{Z}$, the series 
$\Eu_t(\phi)$ does not depend on a choice of 
$t\in (\phi+k, \phi+k+1)$, In particular, we may write 
it as $\Eu^k(\phi)$. 

(ii) The series $\Eu^k(\phi)$ is obtained from 
$\Eu^{k-1}(\phi)$ by the following transformation, 
\emph{\begin{align*}
z \mapsto \left\{ \begin{array}{cc}
(1+x)z, & \mbox{ if }k \mbox{ is even, } \\
xz/(1+x), & \mbox{ if }k\mbox{ is odd. }
\end{array} \right. 
\end{align*}}
\end{thm}
\begin{rmk}
Since the Behrend function is not involved
in the definition of $\Eu_t(r, \beta, n, \phi)$, 
we do not 
rely on Conjecture~\ref{conjBG}
to show Theorem~\ref{thm:Eu}. 
\end{rmk}

\section{Examples}\label{sec:Ex}
In this section, we explicitly compute 
the generating 
series $\DT^{k}(\phi)$ in 
some concrete examples. 
We also classify semistable objects in 
$\dD_X$ w.r.t. our weak stability 
conditions in these examples. 
\subsection{D0-D6 states}\label{subsec:D0D6}
In this subsection, we investigate the 
family of generating series, 
\begin{align}\label{family}
\DT^{k}(\phi), \quad k\in \mathbb{Z}, \quad \phi \in \mathbb{Z}. 
\end{align}
Note that 
a series in (\ref{family})
does not contain 
the variable $y$, since we have
\begin{align*}
\DT_t(r, \beta, n, 1)=0, \quad \mbox{ if }\beta\neq 0,
\end{align*}
which follows from 
$\Imm Z_t(r, -\beta, -n) \neq 0$ 
if $\beta \neq 0$. 

By Theorem~\ref{thm:wall} and the relation (\ref{general}),
it is enough to study the 
series $\DT^{-1}(1)$
to know all of the series (\ref{family}).
By definition, $\DT^{-1}(1)=\DT_t(1)$
for $0<t<1$, and 
$\DT_t(1)$ is a generating series of invariants
counting objects $E\in \pP_t(1)$.  
Such objects can be described in the following way.  
\begin{lem}\label{incase}
For $0<t<1$, we have 
\begin{align*}
\pP_t(1)=\Coh_{0}(X)[-1]. 
\end{align*}
\end{lem}
\begin{proof}
Let us take an object $E\in \pP_t(1) \subset \aA$. 
Since $\ch_2(E)=0$, we have
\begin{align}\notag
E\in  & \aA \cap \langle \oO_X, \Coh_0(X) \rangle_{\tr}, \\
\label{lastca}
& =\langle \oO_X, \Coh_{0}(X)[-1] \rangle_{\ex}. 
\end{align}
By~\cite[Proposition~2.2]{Trk2}, 
any object in the category (\ref{lastca})
is isomorphic to a 
two term complex, 
\begin{align*}
\cdots \to 0 \to \oO_X^{\oplus r} \to F \to 0 \to \cdots,
\end{align*}
where $r\in \mathbb{Z}_{\ge 0}$, $F\in \Coh_{0}(X)$
and $\oO_X^{\oplus r}$ is located in degree zero. 
In particular there is an exact sequence in $\aA$, 
\begin{align*}
0 \to F[-1] \to E \to \oO_X^{\oplus r} \to 0,
\end{align*}
for some $F\in \Coh_0(X)$. 
If $r\neq 0$ and $F\neq 0$, then we have 
\begin{align*}
\pi=\arg Z_t(F[-1]) >\arg Z_t(\oO_X^{\oplus r}) =\pi t, 
\end{align*}
which contradicts to the $Z_t$-semistability of $E$. 
Therefore we have $r=0$ or $F=0$. 
If $F=0$, then 
$E\in \langle \oO_X \rangle_{\ex}
\subset \pP_t(t)$
by Lemma~\ref{below}, which 
contradicts to $E\in \pP_t(1)$. 
Hence $r=0$ and $E\in \Coh_0(X)[-1]$
follows.  

Conversely if $E\in \Coh_0(X)[-1]$, the $Z_t$-semistability 
of $E$ follows from the fact that $\Coh_0(X)[-1] \subset \aA$
is closed under subobjects and quotients. 
\end{proof}
The above lemma and the computations in~\cite[Paragraph~6.3]{JS}, 
\cite[Paragraph~6.5]{K-S}, \cite[Remark~8.13]{Tcurve1}
show that 
\begin{align*}
\DT^{-1}(0, 0, n, 1)=
-\chi(X) \sum_{m\ge 1, m|n} \frac{1}{m^2}, 
\end{align*}
and $\DT^{-1}(r, \beta, n, 1)=0$ if $(r, \beta)\neq (0, 0)$. 
Therefore we have 
\begin{align*}
\DT^{-1}(1)= -\chi(X)
\sum_{\begin{subarray}{c}
n\ge 1, m\ge 1, \\
m|n \end{subarray}}
\frac{1}{m^2} z^n. 
\end{align*}
Applying Theorem~\ref{thm:wall}, 
we obtain the following. 
\begin{thm}\label{thm:DT}
For $k\in \mathbb{Z}$, we have 
\begin{align*}
\DT^{2k-1}(1) &=
-\chi(X)
\sum_{\begin{subarray}{c}
n\ge 1, m\ge 1, \\
m|n \end{subarray}}
\frac{1}{m^2} x^{kn} z^n, \\
 \DT^{2k}(1) &=
-\chi(X)
\sum_{\begin{subarray}{c}
n\ge 1, m\ge 1, \\
m|n \end{subarray}}
\frac{1}{m^2} (1-(-1)^n x)^n x^{kn} z^n.
\end{align*}
\end{thm}
Let us consider the series 
\begin{align*}
\DT^{0}(1)=-\chi(X)
\sum_{\begin{subarray}{c}
n\ge 1, m\ge 1, \\
m|n \end{subarray}}
\frac{1}{m^2} (1-(-1)^n x)^n z^n.
\end{align*}
This is a generating series of 
invariants counting $E\in \pP_t(1)$
for $1<t<2$. The following lemma
shows that such objects  
are certain two term 
complexes
$(\oO_X^{\oplus r} \stackrel{s}{\to} F)$
 with $F\in \Coh_0(X)$.  
\begin{lem}\label{count}
For $1<t<2$, an object $E\in \dD_X$
is contained in $\pP_t(1)$ if and only 
if $E$ is isomorphic to a 
two term complex, 
\begin{align}\label{term}
\cdots \to 0 \to \oO_X^{\oplus r} \stackrel{s}{\to} F \to 0 \to \cdots,
\end{align}
where $\oO_X^{\oplus r}$ is located in 
degree zero and $F\in \Coh_0(X)$, such that the induced morphism 
\begin{align}\label{Hinj}
H^0(s) \colon \mathbb{C}^{r} \to H^0(X, F),
\end{align}
is injective. 
\end{lem}
\begin{proof}
Let us take $E\in \pP_{t}(1)$
for $1<t<2$. 
We see that $E$ 
is isomorphic to a two term complex (\ref{term})
such that (\ref{Hinj}) is injective. 
Since $\pP_t(1) \subset \bB_{+}$, there is an exact
sequence in $\bB_{+}$, 
\begin{align*}
0 \to E_1 \to E \to E_2 \to 0,
\end{align*}
such that $E_1 \in \aA_{+}$ and $E_2 \in \langle \oO_X[-1] \rangle_{\ex}$. 
If $E_2 \neq 0$, then $E_1 \neq 0$
and we have  
\begin{align*}
\pi=\arg Z_{t}(E_1)>\arg Z_{t}(E_2)=\pi(t-1),
\end{align*}
which contradicts to the $Z_t$-semistability of $E$. 
Therefore $E_2=0$ and $E\in \aA_{+} \subset \aA$
follows. 
Since $\ch_2(E)=0$,  
the same argument in the proof of Lemma~\ref{incase}
shows that  
the $E$ is isomorphic to a 
two term complex of the form (\ref{term}), 
and we need to see that $H^0(s)$ is injective.  
There is an exact sequence in $\aA$, 
\begin{align}\label{FEO}
0 \to F[-1] \to E \to \oO_X^{\oplus r} \to 0. 
\end{align}
Applying $\Hom(\oO_X, \ast)$, we obtain the exact sequence, 
\begin{align}\label{H0}
0 \to \Hom(\oO_X, E) \to \mathbb{C}^r \stackrel{H^{0}(s)}{\to} H^0(X, F). 
\end{align}
Since $E\in \aA_{+}$, 
we have $\Hom(\oO_X, E)=0$ and 
the morphism $H^0(s)$ is injective. 

Conversely suppose that $E\in \dD_X$ is 
isomorphic to a two term complex
(\ref{term}) such that $H^0(s)$ is injective.
Then $E\in \aA$ and we
 have the same exact sequences (\ref{FEO}), (\ref{H0}). 
Then we have $\Hom(\oO_X, E)=0$ since $H^0(s)$ is injective,
and we have 
$E \in \aA_{+} \subset \bB_{+}$. 
It remains to check that $E$ is $Z_t$-semistable 
in $\bB_{+}$. 
Let us take an exact sequence in $\bB_{+}$, 
\begin{align}\label{E12}
0 \to E_1 \to E \to E_2 \to 0,
\end{align}
with non-zero $E_1, E_2 \in \bB_{+}$. 
Since $\ch_2(E)=0$, we have 
\begin{align}\label{ch20}
\ch_2(E_1)=\ch_2(E_2)=0.
\end{align}
Applying $\hH_{\aA}^{\bullet}$ to (\ref{E12}), 
We have the long exact sequence in $\aA$, 
\begin{align*}
0 \to \hH_{\aA}^{0}(E_1) \to E \to \hH_{\aA}^{0}(E_2) 
\to \hH_{\aA}^1(E_1) \to 0,
\end{align*}
and $\hH_{\aA}^{1}(E_2)=0$. 
Therefore $\hH_{\aA}^{0}(E_2)\neq 0$, and hence 
\begin{align*}
\ch_3(\hH_{\aA}^{0}(E_2))=\ch_3(E_2) \neq 0. 
\end{align*}
This together with (\ref{ch20}) imply that 
\begin{align*}
\arg Z_t(E_1) \le \arg Z_t(E_2) =\pi,
\end{align*}
which shows the $Z_t$-semistability of $E$. 
\end{proof}
Let us consider the rank one generating series. 
The above lemma shows that
the invariant
$\DT^{0}(1, 0, n, 1)$
counts two term complexes, 
\begin{align*}
\cdots \to 0 \to \oO_X \stackrel{s}{\to} F \to 0 \to \cdots, 
\end{align*}
such that $F\in \Coh_0(X)$ is of length $n$
and $s$ is non-zero. 
By Theorem~\ref{thm:DT}, 
the rank one
generating series
satisfies the following curious equality,  
\begin{align}\notag
\sum_{n\ge 0}\DT^{0}(1, 0, n, 1)z^n &=
-\chi(X) \sum_{\begin{subarray}{c}
n\ge 1, m\ge 1, \\
\label{MacM}
m|n \end{subarray}}
\frac{1}{m^2}(-1)^{n-1}nz^n, \\
&= \log 
M(-z)^{\chi(X)}. 
\end{align}
Here $M(z)$ is the MacMahon function, 
\begin{align*}
M(z)=\prod_{n\ge 1}\frac{1}{(1-z^n)^n}. 
\end{align*}
Recall that $M(-z)^{\chi(X)}$ is 
the generating series of 
DT invariants counting ideal sheaves of points. 
(cf.~\cite{Li}, \cite{BBr}, \cite{LP}.)

\subsection{D0-D2-D6 states on a $(-1, -1)$-curve}\label{-1-1}
Let 
\begin{align*}
f \colon X \to Y,
\end{align*}
be a crepant small resolution of an ordinary 
double point $p \in Y$, 
\begin{align*}
\widehat{\oO}_{Y, p} \cong
\mathbb{C}
\db[ x_1, x_2, x_3, x_4 \db] /(x_1 x_2+x_3 x_4).
\end{align*}
The exceptional locus
 $C \subset X$
satisfies that 
\begin{align*}
C \cong \mathbb{P}^1, \quad N_{C/X} \cong
\oO_{C}(-1) \oplus \oO_C(-1). 
\end{align*}
Let $\Coh_C(X)$ be 
\begin{align*}
\Coh_C(X)\cneq \{ 
E\in \Coh(X) : \Supp(E)\subset C\},
\end{align*}
and we define the triangulated category $\dD_{X/Y}$ to be
\begin{align*}
\dD_{X/Y} \cneq \langle \oO_X, \Coh_C(X) \rangle_{\tr}
\subset \dD_X. 
\end{align*}
Similarly to the case of $\dD_X$, we
 can consider the space of weak 
stability conditions on $\dD_{X/Y}$. 
The required data is as follows. 
We set $\Gamma'$ to be 
\begin{align*}
\Gamma' \cneq H^0(X, \mathbb{Z})\oplus 
\mathbb{Z}[C] \oplus H_0(X, \mathbb{Z}), 
\end{align*}
and define $\cl \colon K(\dD_{X/Y}) \to \Gamma'$
to be 
\begin{align*}
\cl(E)=(\ch_0(E), \ch_2(E), \ch_3(E)),
\end{align*}
for $E\in \dD_{X/Y}$. 
Also we choose the filtration $\Gamma_{\bullet}'$ as 
\begin{align*}
\Gamma_{0}'\cneq H^0(X, \mathbb{Z}) \subset 
\Gamma_{1}'\cneq \Gamma'. 
\end{align*}
The associated space of weak stability 
conditions is denoted by 
$\Stab_{\Gamma'_{\bullet}}(\dD_{X/Y})$. 
Similarly to the map (\ref{map:g}), 
we can construct a continuous map,
\begin{align*}
\gamma' \colon 
\mathbb{R} \to \Stab_{\Gamma'_{\bullet}}(\dD_{X/Y}),
\end{align*}
such that if $t\in (k, k+1]$ for $k\in \mathbb{Z}$, we have 
\begin{align*}
\gamma'(t)=(Z_{(\exp(\pi it), i\omega)}', \aA_{k}'). 
\end{align*}
Here 
\begin{align*}
Z_{(\exp(\pi it), i\omega)}' \in 
\prod_{i=0}^{1} \Hom_{\mathbb{Z}}(\Gamma_i'/\Gamma_{i-1}', \mathbb{C}),
\end{align*}
is defined in a similar way to (\ref{Z1}), (\ref{Z2})
for the fixed ample divisor $\omega$ on $X$, and 
$\aA_{k}'$ are hearts of bounded t-structures, satisfying 
the following. 
\begin{itemize}
\item For $k=0$, we have 
\begin{align*}
\aA_{0}'= \langle \oO_X, \Coh_C(X)[-1] \rangle_{\ex}. 
\end{align*}
\item For $k=1$, we have 
\begin{align*}
\aA_{1}' =\langle \aA_{0, +}', \oO_X[-1] \rangle_{\ex}, 
\end{align*}
where $\aA_{0, +}'=\{E\in \aA_{0}' : \Hom(\oO_X, E)=0\}$. 
\item The other $\aA_k'$ are determined by the rule, 
\begin{align*}
\aA_{k+2}'=\Phi_{\oO_X}\aA_k'. 
\end{align*}
\end{itemize}
Let us write 
\begin{align*}
\gamma'(t)=(Z_t', \pP_t'), \quad t\in \mathbb{R}, 
\end{align*}
for a slicing $\pP_t'$ on $\dD_{X/Y}$. 
Similarly to Definition~\ref{def:around}, 
we can define the DT type invariants, 
\begin{align}\label{simDT}
\DT_t(r, m, n, \phi) \in \mathbb{Q}, 
\end{align}
counting objects $E\in \dD_{X/Y}$ 
satisfying 
\begin{align*}
E\in \pP_{t}'(\phi), \quad \cl(E)=(r, -m[C], -n).
\end{align*}
By abuse of notation, 
we define the generating series as well, 
\begin{align*}
\DT_{t}(\phi)=\sum_{(r, m[C], n) \in \Gamma'}
 \DT_t(r, m, n, \phi)x^r y^m z^n. 
\end{align*}
A result similar to Theorem~\ref{thm:wall} also holds
as follows. 
The proof is similar and we omit the proof. 
\begin{thm}\label{odp:thm}
(i) For a given $k\in \mathbb{Z}$, the series
$\DT_{t}(\phi)$ does not depend on a choice 
of $t\in (\phi+k, \phi+k+1)$. 
In particular, we may write it as $\DT^{k}(\phi)$. 

(ii) The series $\DT^{k}(\phi)$ is obtained from $\DT^{k-1}(\phi)$
by the following transformation, 
\emph{\begin{align*}
z^n \mapsto \left\{ \begin{array}{cc}
(1-(-1)^{n}x)^{n}z^n, & \mbox{ if }k \mbox{ is even.} \\
x^n z^n/(1-(-1)^n x)^{n}, & \mbox{ if }k\mbox{ is odd.}
\end{array} \right. 
\end{align*}}
\end{thm}
Let us investigate what kinds of objects 
the invariants $\DT_t(r, m, n, \phi)$ count. 
Recall that there is the heart of a bounded 
t-structure, called the \textit{perverse
t-structure}~\cite{Br1}, \cite{MVB},
\begin{align*}
\mathrm{Per}(X/Y) \subset D^b(\Coh_C(X)). 
\end{align*}
Its generator is given in~\cite[Proposition~3.5.7]{MVB},
\begin{align}\label{def:Per}
\mathrm{Per}(X/Y)=\langle \oO_C(-1)[1], \oO_C \rangle_{\ex}. 
\end{align}
We have the following lemma. 
\begin{lem}\label{forward}
For $1/2 <t \le 3/2$, we have 
\begin{align}\label{Per}
\pP_t'((1/2, 3/2])=\langle \oO_X,  \mathrm{Per}(X/Y)[-1]
\rangle_{\ex}. 
\end{align}
\end{lem}
\begin{proof}
The argument of~\cite[Lemma~3.2 (ii)]{Tcurve2}
shows that the RHS side
is the heart of a 
bounded t-structures on $\dD_{X/Y}$. 
Hence it is enough to check that 
the RHS is contained in the LHS. 
It is straightforward to check that 
\begin{align*}
\{\oO_X, \oO_C(-1), \oO_C[-1] \} \subset \pP_t'((1/2, 3/2]). 
\end{align*}
Since $\oO_X$, $\oO_C(-1)$ and $\oO_C[-1]$
generates the RHS of (\ref{Per}), 
we obtain the result. 
\end{proof}
By the description (\ref{def:Per}), we can easily see that 
\begin{align*}
\Hom(\oO_X, E[i])=0, \quad i>0, 
\end{align*}
for any $E\in \mathrm{Per}(X/Y)$. 
Then the same argument of~\cite[Proposition~2.2]{Trk2}
shows that the RHS of (\ref{Per})
is equivalent to the abelian category of triples, 
\begin{align}\label{Percoh}
(\oO_X^{\oplus r}, F, s), 
\end{align}
where $r\in \mathbb{Z}_{\ge 0}$, 
$F\in \mathrm{Per}(X/Y)$ and 
$s$ is a morphism, $s \colon \oO_X^{\oplus} \to F$.
The set of morphisms 
are given by the commutative diagram, 
\begin{align*}
\xymatrix{
\oO_X^{\oplus r_1} \ar[r]^{s_1} \ar[d]_{g} & F_1 \ar[d]_{h} \\
\oO_X^{\oplus r_2} \ar[r]^{s_2} & F_2, 
}
\end{align*}
and the equivalence is given by 
sending the triple (\ref{Percoh})
to the total complex of the double complex
$(\oO_X^{\oplus r} \stackrel{s}{\to} F)$.
The above category of triples (\ref{Percoh})
is nothing but the category of \textit{perverse 
coherent systems} considered in~\cite{NN}. 
Noting this, we have the following lemma. 
\begin{lem}\label{kinds}
For $1/2< t<\phi \le 3/2$
with $\phi \neq 1$,  
we have $\pP_t'(\phi) \neq \{0\}$
if and only
there is $a\in \mathbb{Z}$
satisfying
\begin{align}\label{Note}
-a +(\omega \cdot C) \sqrt{-1} \in \mathbb{R} \exp(i\pi \phi). 
\end{align}
In this case, we have 
\emph{\begin{align}\label{inthis}
\pP_t'(\phi)=\left\{ 
\begin{array}{cc}
\langle \oO_C(a-1)[-1] \rangle_{\ex}, & \mbox{ if }1/2 <\phi<1, \\
\langle \oO_C(a-1) \rangle_{\ex}, & \mbox{ if }1<\phi<3/2.
\end{array}
\right. 
\end{align}}
\end{lem}
\begin{proof}
Let us take a non-zero object $E\in \pP_t'(\phi)$. 
Note that we have 
\begin{align*}
\pP_t'(\phi) \subset \pP_t'((1/2, 3/2]),
\end{align*}
and $E$ is  
a semistable object 
in $\pP_t'((1/2, 3/2])$
with respect to the 
rotated weak stability condition, 
\begin{align*}
\frac{1}{2} \cdot 
\gamma'(t)=
(-i Z_t', \pP_t'((1/2, 3/2])). 
\end{align*}
By Lemma~\ref{forward}
and the subsequent argument, 
any object in $\pP_t'((1/2, 3/2]))$ is
isomorphic to the total complex associated to 
a triple (\ref{Percoh}). 
Hence 
  there is an exact sequence
in $\pP_t'((1/2, 3/2])$, 
\begin{align*}
0 \to F[-1] \to E \to \oO_X^{\oplus r} \to 0, 
\end{align*}
for $r\in \mathbb{Z}_{\ge 0}$ 
and $F\in \mathrm{Per}(X/Y)$. 
Suppose that $r\neq 0$. Then we have 
\begin{align*}
\pi \phi= \arg Z_t'(F[-1]) > \arg Z_t'(\oO_X^{\oplus r})=
\pi t, 
\end{align*}
which contradicts to the $-iZ_t'$-semistability of $E$. 
Here we have taken the arguments in $(\pi/2, 3\pi/2]$. 
Therefore we have $r=0$
or $F=0$. 
If $F=0$, then 
$E\in \langle \oO_X \rangle_{\ex} \subset \pP_t'(t)$
which contradicts to $E\in \pP_t'(\phi)$. 
Therefore $r=0$ and $E\in \mathrm{Per}(X/Y)[-1]$
follows. 

Let $W\colon K(\dD_{X/Y}) \to \mathbb{C}$
be the group homomorphism defined by, 
\begin{align}\label{defW}
W(E)=-\ch_3(E)+i\omega \cdot \ch_2(E). 
\end{align}
Then the pair 
\begin{align}\label{pair}
(W, \Coh_C(X)),
\end{align}
is a Bridgeland's stability condition on $D^b(\Coh_C(X))$, 
and the set of $W$-(semi)stable objects
in $\Coh_C(X)$ coincides
with the set of $\omega$-(semi)stable sheaves in 
$\Coh_C(X)$. 
Let us write
the stability condition (\ref{pair}) as the pair 
$(W, \qQ)$ for a slicing $\qQ$
in $D^b(\Coh_C(X))$,  
and consider the rotated stability condition 
\begin{align}\label{rotate}
\left(-\frac{1}{2} \right) \cdot (W, \Coh_C(X))=
(iW, \qQ((-1/2, 1/2])). 
\end{align}
It is easy to see that 
$\oO_C(-1)$ and $\oO_C[-1]$ are contained 
in $\qQ((-1/2, 1/2])$, hence we have 
\begin{align*}
\qQ((-1/2, 1/2])=\mathrm{Per}(X/Y)[-1]. 
\end{align*}
Under the above identification, 
the set of $iW$-(semi)stable objects in $\qQ((-1/2, 1/2])$
coincides with that of 
$-iZ_t'$-(semi)stable objects in $\mathrm{Per}(X/Y)[-1]$. 
Since an $\omega$-stable sheaf in $\Coh_C(X)$
 is of the form 
$\oO_C(a)$ or $\oO_x$ for $x\in C$,
the set of $-iZ_t'$-stable objects
in $\mathrm{Per}(X/Y)[-1]$
is given as follows, 
\begin{align*}
\{ \oO_C(a-1) : a \le 0\} \cup
\{ \oO_C(a-1)[-1] : a\ge 1 \} \cup 
\{ \oO_x[-1] : x\in C \}. 
\end{align*}
Since we have
\begin{align*}
Z_t'(\oO_C(a-1))=a-(\omega \cdot C)\sqrt{-1}, 
\end{align*}
there is non-zero $E\in \pP_t'(\phi)$ only if 
the condition (\ref{Note}) is satisfied, 
and in this case
 $\pP_t'(\phi)$ is given by (\ref{inthis}).  
\end{proof}
In the situation of Lemma~\ref{kinds}, 
any non-zero object $E\in \pP_t'(\phi)$ is written as 
\begin{align*}
E\cong \left\{ \begin{array}{cc}
\oO_C(a-1)^{\oplus m}[-1], & \mbox{ if }1/2<\phi<1, \\
\oO_C(a-1)^{\oplus m}, & \mbox{ if }1<\phi<3/2, 
\end{array} \right. 
\end{align*}
for some $m\in \mathbb{Z}_{\ge 1}$, noting that 
\begin{align*}
\Ext_{X}^1(\oO_C, \oO_C)=0. 
\end{align*} 
Then a computation similar to Lemma~\ref{already} 
shows that  
\begin{align*}
\DT^{-1}(r, m, n, \phi)=
\left\{ 
\begin{array}{cc}
\frac{1}{m^2}, & \mbox{ if }r=0, n=ma, m\ge 1, \\
0, & \mbox{ otherwise, }
\end{array}
\right. 
\end{align*}
for $1/2<\phi<1$ and 
\begin{align*}
\DT^{0}(r, m, n, \phi) &=
\DT^{-1}(-r, -m, -n, \phi+1) \\
&=\left\{ 
\begin{array}{cc}
\frac{1}{m^2}, & \mbox{ if }r=0, n=ma, m \ge 1, \\
0, & \mbox{ otherwise, }
\end{array}
\right. 
\end{align*}
for $0<\phi \le 1/2$. 
Applying Theorem~\ref{odp:thm}, we obtain the following. 
\begin{thm}\label{thm:ODP}
For $0<\phi<1$, suppose that 
there is $a \in \mathbb{Z}$ 
satisfying (\ref{Note}). 
For $k\in \mathbb{Z}$, we obtain the following. 

(i) If $0<\phi<1/2$, we have 
\begin{align*}
\DT^{2k-1}(\phi) &=
\sum_{m \ge 1}\frac{1}{m^2}x^{kma}y^{m}z^{ma}, \\
\DT^{2k}(\phi) &=
\sum_{m \ge 1}\frac{1}{m^2} (1-(-1)^{ma}x)^{ma}x^{kma}y^{m}z^{ma}. 
\end{align*}

(ii) If $1/2 \le \phi<1$, we have 
\begin{align*}
\DT^{2k}(\phi) &=
\sum_{m \ge 1}\frac{1}{m^2}x^{kma}y^{m}z^{ma}, \\
\DT^{2k+1}(\phi) &=
\sum_{m \ge 1}\frac{1}{m^2} (1-(-1)^{ma}x)^{ma}x^{kma}y^{m}z^{ma}. 
\end{align*}
\end{thm}
Similarly to Lemma~\ref{count}, we 
can investigate what kinds of objects
the invariants (\ref{simDT}) count. 
The case of $\phi -1 \ll t <\phi$
is already studied in Lemma~\ref{kinds}.
The case  
after crossing the wall $t=\phi$, 
i.e. the case of 
$\phi<t \ll \phi+1$ is 
given as follows.   
\begin{lem}\label{lem:P'}
For $1/2<\phi<t\le 3/2$
with $\phi \neq 1$, we have 
$\pP_t(\phi)\neq \{0\}$
if and only if
there is $a\in \mathbb{Z}$
satisfying (\ref{Note}).
In this case, 
an object $E\in \dD_X$ is contained 
in $\pP_t(\phi)$ if and only if 
the following holds.  
\begin{itemize}
\item If $1/2<\phi<1$, 
then $E$ is quasi-isomorphic to 
a two term complex, 
\begin{align}\label{comp1}
\cdots \to 0\to 
\oO_X^{\oplus r}  \stackrel{s}{\to} \oO_C(a-1)^{\oplus m}
\to 0 \to \cdots, 
\end{align}
where $\oO_X^{\oplus r}$ is located in degree zero,
such that the induced morphism, 
\begin{align*}
H^0(s) \colon \mathbb{C}^r \to H^0(C, \oO_C(a-1))^{\oplus m},
\end{align*}
is injective. 
\item 
If $1<\phi<3/2$, then $E$ fits into the 
exact sequence of sheaves, 
\begin{align}\label{comp2}
0 \to \oO_C(a-1)^{\oplus m} \to E \to \oO_X^{\oplus r} \to 0, 
\end{align}
such that the induced morphism, 
\begin{align*}
\mathbb{C}^{r} \to H^{1}(C, \oO_C(a-1))^{\oplus m},
\end{align*}
is injective. 
\end{itemize} 
\end{lem}
\begin{proof}
Let us take an object $E\in \pP_t'(\phi)$. 
By Lemma~\ref{forward}
and the subsequent argument, 
$E$ is isomorphic to 
the total complex of a double complex, 
\begin{align}\label{double}
\oO_X^{\oplus r} \stackrel{s}{\to} F, 
\end{align}
for some $r\ge 0$ and $F\in \mathrm{Per}(X/Y)$. 
In particular,
we have the exact sequence in $\pP_t'((1/2, 3/2])$, 
\begin{align*}
0 \to F[-1] \to E \to \oO_X^{\oplus r} \to 0.
\end{align*}
Applying $\Hom(\oO_X, \ast)$, we obtain the 
exact sequence, 
\begin{align*}
0 \to \Hom(\oO_X, E) \to \mathbb{C}^{r} \stackrel{H^0(s)}{\to}
 \Hom(\oO_X, F). 
\end{align*}
Since $\oO_X \in \pP_t'(t)$, 
$E\in \pP_t'(\phi)$ and $t>\phi$, 
we have $\Hom(\oO_X, E)=0$
hence
the map $H^0(s)$ is injective.

Next we classify the objects $F\in \mathrm{Per}(X/Y)$
which appear in (\ref{double}).  
Let $0\neq F' \subset F$ be a 
subobject in $\mathrm{Per}(X/Y)$. 
Then we have 
injections 
\begin{align*}
F'[-1] \hookrightarrow F[-1] \hookrightarrow 
E,
\end{align*}
in $\pP_t'((1/2, 3/2])$. 
As in the proof of Lemma~\ref{kinds}, we consider
rotated weak stability condition on $\dD_{X/Y}$, stability 
condition on $D^b(\Coh_C(X))$ respectively, 
\begin{align*}
(-iZ_t', \pP_t'((1/2, 3/2])), \\
(iW, \mathrm{Per}(X/Y)[-1]). 
\end{align*}
Here $W$ is given by (\ref{defW}). 
By the $-i Z_t'$-semistability of $E$, 
We have the inequality in $(\pi/2, 3\pi/2]$,
\begin{align*}
\arg W(F') &= \arg Z_t'(F'[-1]) \\
& \le \arg Z_t'(E) \\
&= \arg Z_t'(F[-1]) \\
&=\arg W(F). 
\end{align*}
Therefore $F[-1]$ is an $iW$-semistable 
object in $\mathrm{Per}(X/Y)[-1]$. 
As in the proof of Lemma~\ref{kinds}
and the subsequent argument, 
 there is $m \ge 1$ such that 
\begin{align*}
F \cong \left\{ \begin{array}{cc}
\oO_C(a-1)^{\oplus m}, & \mbox{ if }1/2<\phi <1, \\
\oO_C(a-1)^{\oplus m}[1], & \mbox{ if }1<\phi<3/2. 
\end{array} \right.
\end{align*}
Therefore 
$E$ is isomorphic to a two 
term complex (\ref{comp1})
when $1/2<\phi<1$, and 
isomorphic to a sheaf which fits into 
the exact sequence (\ref{comp2})
when $1<\phi<3/2$. 

Conversely suppose that $E\in \dD_X$
is an object given by (\ref{comp1}) or (\ref{comp2}). 
Then $E$ is the total complex of
a double complex (\ref{double})
for some $F\in \mathrm{Per}(X/Y)$, 
which satisfies the property 
that $F[-1]\in \mathrm{Per}(X/Y)[-1]$ is $iW$-semistable
and $H^{0}(s)$ is injective. 
In particular $E$ is an object in $\pP_t'((1/2, 3/2])$
by Lemma~\ref{forward}. 
Let us take an exact sequence
\begin{align*}
0 \to E_1 \to E \to E_2 \to 0,
\end{align*}
 in $\pP_t'((1/2, 3/2])$
with non-zero $E_1$ and $E_2$. 
The above exact sequence 
corresponds to an exact sequence of 
triples (\ref{Percoh}), 
\begin{align*}
0\to (\oO_X^{\oplus r_1} \stackrel{s_1}{\to}
F_1) \to (\oO_X^{\oplus r} \stackrel{s}{\to}
F) \to  (\oO_X^{\oplus r_2} \stackrel{s_2}{\to}
F_2) \to 0. 
\end{align*}
Since $H^{0}(s)$ is injective, we have $F_{1}\neq 0$.
If we also have $F_2 \neq 0$, 
then the $iW$-semistability of $F[-1]$ implies  
the inequality in $(\pi/2, 3\pi/2]$, 
\begin{align*}
\arg Z_t'(E_1) &=\arg W(F_1) \\
&\le \arg W(F_2) \\
& = \arg Z_t'(E_2).  
\end{align*}
If $F_2=0$, we have 
\begin{align*}
\pi \phi= \arg Z_t'(E_1) <\arg Z_t'(E_2)=\pi t. 
\end{align*}
The above inequalities show that $E$
is $-i Z_t'$-semistable in $\pP_t'((1/2, 2/3])$,
and $E\in \pP_t'(\phi)$ follows.  
\end{proof}
\begin{rmk}
In the case of $\phi=1$, 
the generating series and the relevant 
semistable objects are described in a 
way similar to the results in
 Subsection~\ref{subsec:D0D6}. 
\end{rmk}
In a similar way to (\ref{MacM}), 
Theorem~\ref{thm:ODP} can be
used to write down the rank 
one generating series, 
\begin{align}\notag
\sum_{\begin{subarray}{c}
\phi\in (0, 1), \\
(m, n)\in \mathbb{Z}^{\oplus 2}.
\end{subarray}
}\DT^{0}(1, m, n, \phi)y^m z^n
&=\sum_{\begin{subarray}{c}
a\ge 1, \\
\label{curious}
m\ge 1. \end{subarray}}
\frac{1}{m^2}(-1)^{ma-1} ma y^m z^{ma}, \\
&= \log \prod_{m\ge 1}(1-(-1)^{m}yz^m)^m. 
\end{align}
By Lemma~\ref{lem:P'}, the above
series is a generating series of invariants
counting two term complexes of the form, 
\begin{align*}
\oO_X \stackrel{s}{\to} \oO_C(a-1)^{\oplus m},
\end{align*} 
such that that $s$ is non-zero. 
Here we have again observed a curious 
phenomena, since 
 the series 
\begin{align}\label{nonconn}
 \prod_{m\ge 1}(1-(-1)^{m}yz^m)^m
\end{align}
coincides with the generating series of 
stable pairs on a local $(-1, -1)$-curve. 
(cf.~\cite{PT}). 
Under the GW/DT/PT correspondence~\cite{MNOP}, 
the series (\ref{curious})
corresponds to the connected GW theory, 
while the series (\ref{nonconn}) corresponds 
to the non-connected GW theory~\cite{BeBryan}.  

\begin{rmk}
Although we rely on Conjecture~\ref{conjBG}
to prove Theorem~\ref{thm:wall}, 
we can check a version of 
Conjecture~\ref{conjBG}
needed in showing Theorem~\ref{thm:DT} and Theorem~\ref{thm:ODP}
by hand, following the same strategy of~\cite[Proposition~2.12]{Trk2}. 
Therefore the results in this subsection 
are completely rigorous. 
\end{rmk}

\section{Some technical lemmas}\label{sec:tech}
\subsection{Proof of Lemma~\ref{lem:cons}}
For simplicity, we give a 
proof for the pair $(Z_u, \bB_{+})$
with $u=(z, B+i\omega)
\in (-\mathbb{H})\times A(X)_{\mathbb{C}}$. 
We divide the proof into 3 steps. 
\begin{step}\label{S1}
The pair $(Z_u, \bB_{+})$ satisfies the 
Harder-Narasimhan property. 
\end{step}
\begin{proof}
By~\cite[Proposition~2.12]{Tcurve1}, 
it is enough to check that 
the following conditions are 
satisfied. 
\begin{itemize}
\item The abelian category $\bB_{+}$
is noetherian. 
\item  There are no infinite sequences of 
subobjects in $\bB_{+}$, 
\begin{align}\label{chain1}
\cdots E_{j+1} \subset E_j \subset \cdots \subset E_2 \subset E_1, 
\end{align}
with the following inequality for all $j$, 
\begin{align}\label{ineqZ}
\arg Z_u(E_{j+1}) > \arg Z_{u}(E_j/E_{j+1}).
\end{align} 
\end{itemize}
First we show
that the abelian category $\bB_{+}$ is noetherian. 
For an object $E\in \bB_{+}$, suppose that there 
is an infinite sequence of inclusions in $\bB_{+}$, 
\begin{align}\label{seq:F}
F_1 \subset F_2 \subset \cdots \subset E. 
\end{align}
Applying $\hH_{\aA}^{\bullet}$, 
we obtain the sequence of inclusions in $\aA$, 
\begin{align*}
\hH_{\aA}^{0}(F_1) \subset \hH_{\aA}^{0}(F_2) \subset \cdots 
\subset \hH_{\aA}^{0}(E). 
\end{align*}
Since $\aA$ is noetherian by~\cite[Lemma~6.2]{Trk2}, 
we may assume that $\hH_{\aA}^{0}(F_i)
 \stackrel{\cong}{\to} \hH_{\aA}^{0}(F_{i+1})$
for all $i$. Then taking the quotients
of (\ref{seq:F}) by 
$\hH^{0}_{\aA}(F_1)$, we may assume that 
$\hH^{0}_{\aA}(F_i)=0$ for all $i$. 
This means that we have 
\begin{align*}
F_i \cong \oO_X^{\oplus r_i}[-1], \quad r_i \in \mathbb{Z}_{\ge 0}. 
\end{align*}
Since $r_1 \le r_2 \le \cdots$ and 
$\Hom(\oO_X[-1], E)$ is finite dimensional, 
the sequence (\ref{seq:F}) terminates. 

Next suppose that there is a sequence (\ref{chain1}) 
satisfying (\ref{ineqZ}).  
Note that for any object $E\in \bB_{+}$, we have
$\ch_2(E)\cdot \omega \le 0$ by the description 
of $\bB_{+}$
in Lemma~\ref{types}. 
Therefore we have 
\begin{align*}
\ch_2(E_1) \cdot \omega \le 
\cdots \le \ch_2(E_j) \cdot \omega \le \ch_2(E_{j+1}) \cdot \omega 
\le \cdots \le 0. 
\end{align*}
Hence we may assume that 
$\ch_2(E_j) \cdot \omega=\ch_2(E_{j+1}) \cdot \omega$
for all $j$. 
This implies that $\ch_2(E_j/E_{j+1})=0$, 
and we have either
\begin{align*}
Z_u(E_j/E_{j+1}) \in \mathbb{R}_{<0}, \mbox{ or }
E_j/E_{j+1} \in \langle \oO_X[-1] \rangle_{\ex}. 
\end{align*}
Since we have the inequality (\ref{ineqZ}), 
we have $Z_u(E_j/E_{j+1}) \notin \mathbb{R}_{<0}$. 
Therefore we have $E_j/E_{j+1} \in \langle \oO_X[-1] \rangle_{\ex}$, 
and hence $E_1/E_{j}$ is written as $\oO_X[-1]^{\oplus r_j}$
for some $r_j \in \mathbb{Z}_{\ge 0}$. 
There is a sequence of surjections, 
\begin{align*}
E_1 \twoheadrightarrow
\cdots  \twoheadrightarrow E_1/E_3 \twoheadrightarrow 
 E_1/E_2, 
\end{align*}
hence we have $r_2 \le r_3 \le \cdots$. 
Since $\Hom(E_1, \oO_X[-1])$ is finite dimensional, the above sequence 
must terminate. 
\end{proof}
\begin{step}
The weak stability condition $(Z_{u}, \bB_{+})$ 
satisfies the local finiteness property. 
\end{step}
\begin{proof}
Let $\{\pP(\phi)\}_{\phi \in \mathbb{R}}$ be
the slicing determined by the pair $(Z_u, \bB_{+})$. 
It is enough to check that the following 
quasi-abelian categories, 
\begin{align*}
\pP((0, 1)), \quad \pP((1/2, 3/2)),
\end{align*}
are of finite length. 
The category $\pP((0, 1))$ is contained in 
$\bB_{+}$, and the same argument of Step~\ref{S1}
shows that $\pP((0, 1))$ is of finite length. 
Let us check that $\pP((1/2, 3/2))$ is of 
finite length. 
We take a sequence of strict epimorphisms
in $\pP((1/2, 3/2))$,
(see~\cite[Section~4]{Brs1} for the 
notion of strict epimorphisms and
strict monomorphisms,)
\begin{align}\label{strict}
E_1 \twoheadrightarrow E_2 \twoheadrightarrow \cdots
\twoheadrightarrow E_j \twoheadrightarrow E_{j+1} \twoheadrightarrow \cdots,  
\end{align}
and exact sequences in $\pP((1/2, 3/2))$, 
\begin{align*}
0 \to F_j \to E_{j} \to E_{j+1} \to 0. 
\end{align*}
Note that $\ch_3(E) \le 0$ for any $E\in \pP((1/2, 3/2))$, 
and the inequality is strict if $\cl(E) \notin \Gamma \setminus \Gamma_0$. 
Therefore we have the inequalities, 
\begin{align*}
\ch_3(E_1) \le \cdots \le \ch_3(E_j) \le \ch_3(E_{j+1}) 
\le \cdots \le 0. 
\end{align*}
Hence we may assume that $\ch_3(E_1)=\ch_3(E_j)$,
which implies that $\ch_3(F_j)=0$ for all $j$. 
Then we have $\cl(F_j) \in \Gamma_0$, 
and Lemma~\ref{below} shows that 
\begin{align*}
F_j \cong \left\{ \begin{array}{cc}
\oO_X^{\oplus r_j}, & \mbox{ if }\Ree z<0, \\
0, & \mbox{ if }\Ree z =0, \\
\oO_X[-1]^{\oplus r_j}, & \mbox{ if }\Ree z>0,
\end{array}\right. 
\end{align*}
for some $r_j \in \mathbb{Z}_{\ge 0}$. Then the same 
argument of Step~\ref{S1} shows that the 
sequence (\ref{strict}) terminates, i.e. 
$\pP((1/2, 3/2))$ is noetherian. 
A similar argument also shows that $\pP((1/2, 3/2))$
is artinian with respect to the strict monomorphisms, 
hence $\pP((1/2, 3/2))$ is of finite length. 
\end{proof}
\begin{step}
The pair $(Z_{u}, \bB_{+})$ satisfies the support property. 
\end{step}
\begin{proof}
Let us take a non-zero object $\bB_{+}$, 
and we set $\cl(E)=(r, -\beta, -n)$. 
If $(\beta, n)=(0, 0)$, we have 
\begin{align*}
\frac{\lVert [\cl(E)] \rVert_{0}}{ \rvert Z(E) \rvert}
=\frac{1}{\lvert z \rvert}.  
\end{align*}
Suppose that $(\beta, n)\neq (0, 0)$. 
Then we have 
\begin{align}\label{sup1}
\frac{\lVert [\cl(E)] \rVert_{1}}{ \rvert Z(E) \rvert}
&= \sqrt{\frac{\lVert \beta \rVert^2 +n^2}
{(n-B\cdot \beta)^2 +(\omega \cdot \beta)^2}}.
\end{align}
Here $\lVert \ast \rVert$ is a fixed norm on $H_2 \otimes \mathbb{R}$. 
If $\beta=0$, then (\ref{sup1}) 
equals to $1$. If $\beta \neq 0$, 
then (\ref{sup1}) coincides with 
\begin{align}\label{value}
\sqrt{\frac{1+\mu^2}
{(\mu-B_0)^2 +\omega_0^2}}.
\end{align}
Here we have set
\begin{align*}
\mu=\frac{n}{\lVert \beta \rVert}, \quad 
B_0=B\cdot \frac{\beta}{\lVert \beta \rVert}, 
\quad \omega_0=\omega \cdot \frac{\beta}{\lVert \beta \rVert}. 
\end{align*}
The values $B_0$ and $\omega_0 >0$ are bounded 
w.r.t. non-zero $\beta \in H_2$. 
 Also for fixed $B_0$ and $\omega_0$, 
the value (\ref{value})
is bounded w.r.t. all $\mu \in \mathbb{Q}$. 
Therefore (\ref{value}) is bounded w.r.t.
all $B_0$, $\omega_0$ and $\mu$.   
\end{proof}
\subsection{Proof of Lemma~\ref{omit}}
\begin{proof}
(i) It is easy to see that 
$\Coh_{\le 1}(X)[-1]$ is closed
under subobjects and quotients in 
the abelian category $\aA$. 
Then it is easy to see that 
the $(B, \omega)$-semistability of 
  $F\in \Coh_{\le 1}(X)$
yields the $Z_{u}$-semistability of $F[-1]\in \aA$. 

(ii)
First we suppose that 
\begin{align}\label{byour}\arg Z_{u}(F[-1])>\arg (-z).
\end{align}
Let us take an exact sequence in $\bB_{+}$, 
\begin{align}\label{MN}
0 \to M \to F[-1] \to N \to 0,
\end{align}
with $M, N\neq 0$. 
We want to show the inequality, 
\begin{align}\label{want}
\arg Z_{u}(M) \le \arg Z_{u}(N),
\end{align}
to show the $Z_u$-semistability of $F[-1]$. 
Applying $\hH_{\aA}^{\bullet}$
to the sequence (\ref{MN}), 
we obtain the long exact sequence in $\aA$, 
\begin{align}\label{long}
0 \to \hH_{\aA}^{0}(M) \to F[-1] \stackrel{s}{\to} \hH_{\aA}^{0}(N)
\to \hH_{\aA}^{1}(M) \to 0,
\end{align}
and the vanishing $\hH_{\aA}^1(N)=0$. 
Note that $\hH_{\aA}^{1}(M) \in \langle \oO_X \rangle_{\ex}$
by the construction of $\bB_{+}$. 

Suppose that $\hH_{\aA}^{0}(M)=0$. 
Then we have 
\begin{align*}
\arg Z_{u}(M) &=\arg Z_{u}(\oO_X[-1]) \\
&=\arg (-z). 
\end{align*}
By our assumption (\ref{byour})
this implies the inequality (\ref{want}).

Suppose that $\hH_{\aA}^{0}(M)\neq 0$. 
Then $\hH_{A}^{0}(M)$ and the image of $s$ 
are written as $F'[-1]$, $F''[-1]$ for some
$F', F'' \in \Coh_{\le 1}(X)$ respectively. 
Note that we have 
\begin{align*}
Z_{u}(M)=Z_{u}(F'[-1]), \quad 
Z_{u}(N)=Z_{u}(F''[-1]). 
\end{align*}
The exact sequence (\ref{long}) 
and the $(B, \omega)$-semistability of $F$ yield, 
\begin{align*}
\arg Z_{u}(F'[-1]) \le \arg Z_{u}(F''[-1]). 
\end{align*}
Therefore the inequality (\ref{want}) holds
and $F[-1]$ is $Z_u$-semistable. 

Next suppose that 
\begin{align*}
\arg Z_u(F[-1])<\arg(-z). 
\end{align*}
In this case, a similar argument as above 
shows that $F[-1]$ is a $Z_u$-semistable 
object in $\bB_{-}$. Applying the 
twist functor $\Phi_{\oO_X}$
and using Lemma~\ref{phist}, 
we conclude that $\Phi_{\oO_X}(F[-1]) \in \bB_{+}$
and it is $Z_u$-semistable. 
\end{proof}

\subsection{Proof of Lemma~\ref{lem:dense}}
\begin{proof}
For a non-zero object $F\in \Coh_{\le 1}(X)$, 
 we have
\begin{align*}
Z_u (\Phi_{\oO_X}(F[-1]))=Z_u(F[-1]). 
\end{align*}
Therefore
by Lemma~\ref{omit},
 it is enough to show the density 
of the slope of $(B, \omega)$-semistable 
sheaves. 

Let $D\subset X$ be a sufficiently ample divisor. 
For each $l\ge 1$
and $k\in \mathbb{Z}$, 
we choose the following, 
\begin{align*}
C_l \in \lvert \oO_{D}(lD) \rvert,
\quad L_{l, k} \in \Pic(C_l). 
\end{align*}
Here $C_l$ is a smooth member
and the degree of $L_{l, k}$ is equal to $k$.
Note that $L_{l, k}$ is $(B, \omega)$-stable sheaf
on $X$.  
Setting $C=C_1$
and $d=\int_{X}D^3 \in \mathbb{Z}$, we have 
\begin{align*}
\cl(L_{l, k})=(0, l[C], k- dl(l+1)/2). 
\end{align*}
 Therefore we obtain 
\begin{align*}
Z_{u}(L_{l, k}[-1])=
-k+(B\cdot C)l+\frac{dl(l+1)}{2} 
+(\omega \cdot C)l\sqrt{-1}.
\end{align*}
Then it is easy to see that  
\begin{align*}
\left\{
\frac{\Ree Z_{u}(L_{l, k}[-1])}{\Imm Z_{u}(L_{l, k}[-1])}
: l\ge 1, \ k\in \mathbb{Z} \right\}
=\mathbb{Q}+\left\{ \frac{B\cdot C}{\omega \cdot C}  \right\}. 
\end{align*}
This implies the 
density of (\ref{dense}). 
\end{proof} 

\subsection{Proof of Lemma~\ref{finsum}}
\begin{proof}
By Theorem~\ref{thm:cov}
and Proposition~\ref{prop:re}, we may assume that 
$\sigma \in \uU_{0}$ or $\sigma \in \uU_{1}$. 
For simplicity we show the case of $\sigma \in \uU_{1}$. 
In this case, we can write $\sigma$ as a pair,
\begin{align}\label{uB}
\sigma=(Z_{u}, \bB_{+}), \quad u=(z, B+i\omega) \in (-\mathbb{H}) \times A(X)_{\mathbb{C}},
\end{align}
as in Lemma~\ref{lem:cons} (ii). 
We fix $v=(r, -\beta, -n) \in \Gamma$
and take 
\begin{align*}
m\ge 1, \quad v_i=(r_i, -\beta_i, -n_i), \ 1\le i\le m, 
\end{align*}
which appears in a non-zero term of the RHS of (\ref{sum}). 
We set 
\begin{align*}
m' \cneq \sharp \{ 1\le i \le m : 
(\beta_i, n_i) \neq (0, 0)\}. 
\end{align*}
Note that if $(\beta_i, n_i) \neq (0, 0)$, then 
we have $\beta_i \cdot \omega>0$ or 
$\beta_i=0$, $n_i>0$. 
Also we have 
\begin{align*}
Z_u(v_i)=-n_i +(B+i\omega)\cdot \beta_i,
\end{align*}
if $(\beta_i, n_i)\neq (0, 0)$, 
and they are contained in a same line. 
This implies that $m'$ is bounded above and 
the possibilities of $(\beta_i, n_i)$ are finite. 
By Lemma~\ref{lem:bound} below, 
the $r_i$ is also bounded above, depending only on $(\beta, n)$. 
Therefore the possibilities of $m$ and $v_i$ are finite. 
\end{proof}
We have used the following lemma. 
\begin{lem}\label{lem:bound}
For a fixed $(r, -\beta, -n) \in \Gamma$, the following 
set of objects is bounded,
\emph{\begin{align*}
\left\{ E\in \bB_{+} : \begin{array}{l}
E \mbox{ is }Z_u\mbox{-semistable satisfying} \\
\cl(E)=(r', -\beta, -n), \ r' \ge r. 
\end{array}\right\}.
\end{align*}}
\end{lem}
\begin{proof}
Let us take a $Z_u$-semistable object 
$E\in \bB_{+}$ with $\cl(E)=(r', -\beta, -n)$
for some $r' \ge r$. 
If $(\beta, n)=0$, then $E\in \langle \oO_X[-1] \rangle_{\ex}$
and the result is obvious. We assume that $(\beta, n)\neq (0, 0)$. 
By the construction of $\bB_{+}$ in Lemma~\ref{types} (ii), 
there is an exact sequence in $\bB_{+}$, 
\begin{align}\label{boundE}
0 \to E' \to E \to E'' \to 0, 
\end{align}
such that $E' \in \aA_{+}$ and $E'' \in \langle \oO_X[-1] \rangle_{\ex}$. 
Moreover by the construction of $\aA$ 
and $\aA_{+}$ in Lemma~\ref{types}, 
there is a filtration of $E'$ in $\aA$, 
\begin{align}\label{filtE}
0=E_0 \subset E_1 \subset E_2 \subset \cdots \subset E_N=E', 
\end{align}
such that we have 
\begin{align*}
E_i /E_{i-1} \cong \left\{
\begin{array}{cc}
F_i[-1], & \mbox{ if }i \mbox{ is odd, } \\
\oO_X^{\oplus r_i}, & \mbox{ if }i \mbox{ is even, }
\end{array}
  \right. 
\end{align*}
for some $F_i \in \Coh_{\le 1}(X)$ and $r_i\in \mathbb{Z}_{\ge 1}$. 
Note that $\ch_2(F_i) \cdot \omega \ge 0$, 
hence $\ch_2(F_i)$ and the 
length of the filtration $N$ have finite number 
of possibilities. Let us take the exact sequence
in $\aA$, 
\begin{align*}
0 \to E_i \to E' \to E'/E_{i} \to 0.
\end{align*}
Applying $\hH_{\bB_{+}}^{\bullet}$, we obtain
the exact sequence in $\bB_{+}$, 
\begin{align*}
0\to 
\hH_{\bB_{+}}^{-1}(E'/E_i) \to \hH_{\bB_{+}}^{0}(E_i) 
\stackrel{\iota}{\to} E'. 
\end{align*}
Since $\hH_{\bB_{+}}^{-1}(E/E_i) \in \langle \oO_X \rangle_{\ex}$, 
the $Z_u$-semistability of $E$ implies that 
\begin{align*}
\arg Z_{u}(\oplus_{j\le i}F_j[-1]) &= \arg Z_u(E_i), \\
&=\arg Z_u(\hH_{\bB_{+}}^{0}(E_i)), \\
&= \arg Z_u(\Imm \iota), \\
& \le \arg Z_u(E'), \\
&= \arg Z_u((0, -\beta, -n)),
\end{align*}
for all $i$. 
The above inequality implies that the 
pairs $(\ch_2(F_i), \ch_3(F_i))$ have 
finite number of possibilities. 
By taking Harder-Narasimhan filtrations
of $F_i$ with respect to $\omega$-stability
and applying the same argument, we can also show 
that the Chern characters of Harder-Narasimhan factors 
of each $F_i$ have finite number of 
possibilities. Since $\omega$-semistable sheaves with 
a fixed numerical class is bounded, we conclude that 
possible $F_i$ which appear in the 
filtration (\ref{filtE}) are contained in a bounded family. 

We show by induction on $i$ that the 
possible $E_i$ in the filtration (\ref{filtE}) are 
contained in a bounded family. 
Note that we have already proved the boundedness
for $i=1$. 
Suppose that the claim holds for $i-1$. We have the 
exact sequence in $\aA$, 
\begin{align*}
0 \to E_{i-1} \to E_{i} \to E_{i}/E_{i-1} \to 0. 
\end{align*} 
If $i$ is odd, then $E_{i}/E_{i-1}$ is isomorphic to 
$F_{i}[-1]$ which is contained in a bounded family. 
The object $E_{i-1}$
is also contained in a bounded family by the 
inductive assumption, hence 
 $E_i$ is contained in a bounded family. 
If $i$ is even, then $E_i/E_{i-1}$ is written as 
$\oO_X^{\oplus r_i}$ for some $r_i \in \mathbb{Z}_{\ge 0}$. 
The inductive assumption implies that there is $R>0$ 
 such that we have 
\begin{align*}
\dim \Hom(\oO_X, E_{i-1}[1])\le R,
\end{align*}
for any possible $E_{i-1}$
which appears in (\ref{filtE}). 
Therefore if $r_i>R$, then there is a non-trivial 
morphism $\oO_X \to E_i$ which 
contradicts to $E\in \aA_{+}$. Hence $r_i \le R$ and 
$E_i$ is contained in a bounded family. 

The above argument shows that the object $E'$ in 
(\ref{boundE}) is contained in a bounded family. 
Since $r'\ge r$
and $E'' \in \langle \oO_X[-1] \rangle_{\ex}$,
the boundedness of $E'$ implies the boundedness of 
$E''$. 
 Therefore the object $E$ is contained in a bounded family. 
\end{proof}

\subsection{Proof of Lemma~\ref{const}}
\begin{proof}
As in the same way of the proof of Lemma~\ref{finsum}, 
we may assume that $\sigma \in \uU_1$, hence it is 
written as (\ref{uB}).
First we 
show that $\mM^v(Z_u)$ is a constructible subset in $\mM$. 
For $n\ge 1$, let $\fF il^n(\bB_{+})$ be the 
stack of $n$-step filtrations in $\bB_{+}$. 
Namely a $\mathbb{C}$-valued point 
of $\fF il^n(\bB_{+})$ corresponds to 
a filtration in $\bB_{+}$, 
\begin{align*}
E_1 \subset E_2 \subset \cdots \subset E_n. 
\end{align*}
We have the morphisms of stacks, 
\begin{align*}
p_i \colon \fF il^n(\bB_{+}) \ni E_{\bullet} 
\mapsto E_i/E_{i-1} \in \oO bj(\bB_{+}), 
\end{align*}
and the diagram, 
\begin{align*}
\xymatrix{
\fF il^n(\bB_{+}) \ar[r]^{q_n} \ar[d]_{r_n} & \oO bj(\bB_{+}), \\
\oO bj(\bB_{+})^{\times n}, & 
}
\end{align*}
where $q_n(E_{\bullet})=E_n$ and $r_n=(p_1, \cdots, p_n)$. 
Note that $q_n$ and $r_n$ are
piecewise constructible bundles. 
The proof of Lemma~\ref{lem:bound} shows that 
for each $v\in \Gamma$, there are $N\ge 1$, finite 
subset, 
\begin{align*}
\sS_n \subset \overbrace{\Gamma \times \cdots \times \Gamma}^{n},
\end{align*}
for each $1\le n\le N$ and algebraic substacks of finite type, 
\begin{align*}
\mM^{(v_1, \cdots, v_n)} \subset \oO bj(\bB_{+})^{\times n}, 
\end{align*}
such that we have 
\begin{align*}
\mM^{v}(Z_{u})=\bigcup_{\begin{subarray}{c}1\le n\le N, \\
(v_1, \cdots, v_n) \in \sS_n.
\end{subarray}}
q_n r_n^{-1}(\mM^{(v_1, \cdots, v_n)}). 
\end{align*}
Since the RHS is a finite union of 
constructible subsets, 
 $\mM^v(Z_u)$ is a constructible 
set in $\mM$. 

Next the existence of the Harder-Narasimhan filtration implies 
that 
\begin{align*}
\oO bj(\bB_{+})=
\bigcup_{\begin{subarray}{c} 
n\ge 1, \\
v_1, \cdots, v_n \in \Gamma.
\end{subarray}}
q_n r_n^{-1}
(\mM^{v_1}(Z_u) \times \cdots \times \mM^{v_n}(Z_u)). 
\end{align*}
Since $\mM^{v_i}(Z_u)$ is a constructible subset in $\mM$, 
the RHS is a countable union of constructible subsets in $\mM$. 
\end{proof}

Institute for the Physics and 
Mathematics of the Universe, University of Tokyo

\textit{E-mail address}: yukinobu.toda@ipmu.jp

\end{document}